\documentclass[a4paper,10pt,reqno]{amsart}

\usepackage[utf8]{inputenc}
\usepackage[T1]{fontenc}
\usepackage{lmodern}
\usepackage[english]{babel}
\usepackage{microtype}
\usepackage{pdfsync}

\usepackage{amsmath,amssymb,amsfonts,amsthm}
\usepackage{mathtools,accents}
\usepackage{mathrsfs}
\usepackage{aliascnt}
\usepackage{braket}
\usepackage{bm}
\usepackage{esint}

\usepackage[a4paper,margin=3cm]{geometry}
\usepackage[citecolor=blue, linkcolor=red,colorlinks]{hyperref}

\usepackage{enumerate}
\usepackage{xcolor}
\usepackage{xspace}

\makeatletter
\g@addto@macro\@floatboxreset\centering
\makeatother

\makeatletter
\def\newaliasedtheorem#1[#2]#3{
	\newaliascnt{#1@alt}{#2}
	\newtheorem{#1}[#1@alt]{#3}
	\expandafter\newcommand\csname #1@altname\endcsname{#3}
}
\makeatother

\numberwithin{equation}{section}

\newtheoremstyle{slanted}{\topsep}{\topsep}{\slshape}{}{\bfseries}{.}{.5em}{}

\theoremstyle{plain}
\newtheorem{theorem}{Theorem}[section]
\newaliasedtheorem{proposition}[theorem]{Proposition}
\newaliasedtheorem{lemma}[theorem]{Lemma}
\newaliasedtheorem{corollary}[theorem]{Corollary}
\newaliasedtheorem{counterexample}[theorem]{Counterexample}

\theoremstyle{definition}
\newaliasedtheorem{definition}[theorem]{Definition}
\newaliasedtheorem{question}[theorem]{Question}
\newaliasedtheorem{assumption}[theorem]{Assumption}
\newaliasedtheorem{conjecture}[theorem]{Conjecture}

\theoremstyle{remark}
\newaliasedtheorem{remark}[theorem]{Remark}
\newaliasedtheorem{example}[theorem]{Example}

\newcommand{\setN}{\mathbb{N}}

\newcommand{\setR}{\mathbb{R}}

\newcommand{\eps}{\varepsilon}

\let\altphi\phi
\let\phi\varphi
\let\varphi\altphi
\let\altphi\undefined

\newcommand{\abs}[1]{\left\lvert#1\right\rvert}
\newcommand{\norm}[1]{\left\lVert#1\right\rVert}

\let\div\undefined
\DeclareMathOperator{\div}{div}

\newcommand{\di}{\mathop{}\!\mathrm{d}}

\newcommand{\loc}{{\rm loc}}

\newcommand{\restr}{\raisebox{-.1618ex}{$\bigr\rvert$}}
\DeclareMathOperator{\supp}{supp}

\DeclareMathOperator{\Lip}{Lip}

\newcommand{\haus}{\mathscr{H}}
\newcommand{\leb}{\mathscr{L}}

\newcommand{\XX}{{\boldsymbol{X}}}
\newcommand{\YY}{{\boldsymbol{Y}}}

\newcommand{\sym}{\mathrm{sym}}

\newcommand{\dist}{\mathsf{d}}

\newcommand{\meas}{\mathfrak{m}}

\newcommand{\ap}{\mathsf{ap}}

\newcommand{\Test}{\rm Test}

\DeclareMathOperator{\RCD}{RCD}
\DeclareMathOperator{\ncRCD}{ncRCD}
\newfont{\tmpf}{cmsy10 scaled 2500}

\def\XXint#1#2#3{{\setbox0=\hbox{$#1{#2#3}{\int}$ }
		\vcenter{\hbox{$#2#3$ }}\kern-.6\wd0}}

\begin{document}
	
\title{Improved regularity estimates for Lagrangian flows on $\RCD(K,N)$ spaces}

\author{Elia Bru\'e}  {\thanks{Institute for Advanced Study (Princeton), \url{elia.brue@ias.edu}}}
\author{Qin Deng}     {\thanks{University of Toronto, \url{qin.deng@mail.utoronto.ca}}}
\author{Daniele Semola}  {\thanks{Mathematical Institute, University of Oxford, \url{Daniele.Semola@maths.ok.ac.uk}}}

\maketitle

\begin{abstract}
This paper gives a contribution to the study of regularity of Lagrangian flows on non-smooth spaces with lower Ricci curvature bounds. The main novelties with respect to the existing literature are the better behaviour with respect to time and the local nature of the regularity estimates. These are obtained sharpening previous results of the first and third authors, in combination with some tools recently developed by the second author (adapting to the synthetic framework ideas introduced in \cite{ColdingNaber12}).\\ 
The estimates are suitable for applications to the fine study of $\RCD$ spaces and play a central role in the construction of a parallel transport in this setting.
\end{abstract}


\section{Introduction and main results}

This note deals with regularity estimates for flows of Sobolev velocity fields over non-smooth spaces with synthetic Ricci curvature bounds. With respect to the previous contributions of the first and third author \cite{BrueSemolaahlfors,BrueSemola18} the refinements will be in two directions:
\begin{itemize}
\item a sharper behaviour of the estimates with respect to time;
\item the improvement from infinitesimal estimates to local estimates.
\end{itemize}

\medskip

Flows of vector fields are classically a powerful tool in Partial Differential Equations, Geometric Measure Theory, Differential and Riemannian Geometry. In more recent years, they have turned out to be crucial also in Non Smooth Geometry and Analysis on metric spaces.\\ 
On the one hand, gradient flows of semiconcave functions are fundamental in Alexandrov geometry, see for instance \cite{Petrunin07}. On the other hand, flows of vector fields with integrability rather than uniform bounds on their derivatives are at the core of some developments in the theory of lower Ricci curvature bounds, starting from the seminal \cite{CheegerColding96}.   
\medskip

The framework of our investigation will be that of $\RCD(K,N)$ metric measure spaces, which are a non smooth counterpart of Riemannian manifolds with lower bounds on the Ricci curvature. The $\RCD(K,N)$ class  includes  $N$-dimensional Alexandrov spaces equipped with the Hausdorff measure $\haus^N$ and Ricci limit spaces, i.e. measured Gromov-Hausdorff limits of smooth Riemannian manifolds with lower Ricci curvature bounds. We avoid giving a detailed introduction to this class of spaces and refer the interested reader to the survey paper \cite{Ambrosio18} and references therein.

\subsection*{Vector fields and flow maps on metric measure spaces}

On a metric measure space $(X,\dist,\meas)$ we can understand vector fields as derivations over an algebra of test functions and the divergence operator via integration by parts, see \cite{AmbrosioTrevisan14}.
 In this note we will rely throughout also on the identification of vector fields with elements of the so-called tangent module $L^2(TX)$, referring to \cite{Gigli14} for the relevant background.

As shown in \cite{Gigli14}, there is a second order differential calculus available on $\RCD(K,N)$ spaces (and, more in general, on $\RCD(K,\infty)$ spaces). In particular, the presence of a large class of \emph{regular} test functions $\Test(X,\dist,\meas)$ (see \cite{Savare14,Gigli14}) allows to introduce a natural notion of (time dependent) Sobolev vector field $b\in L^2([0,T];H^{1,2}_{C,s}(TX))$, that we recall below, in the autonomous case for the sake of simplicity.

\begin{definition}\label{def:Sobvect}
		The Sobolev space $H^{1,2}_{C,s}(TX)\subset L^2(TX)$ is the space of all $b\in L^2(TX)$ with $\div b\in L^{2}(X,\meas)$ for which there exists a tensor $S\in L^{2}(T^{\otimes 2}X)$ such that, for any choice of $h,g_1,g_2\in\Test(X,\dist,\meas)$, it holds
		\begin{equation}\label{eq:Sobvectfield}
		\int h S(\nabla g_1,\nabla g_2)\di\meas=\frac{1}{2}\int\left\lbrace -b(g_2)\div(h\nabla g_1)-b(g_1)\div(h\nabla g_2)+\div(hb)\nabla g_1\cdot\nabla g_2\right\rbrace \di\meas.
		\end{equation}
		In this case we shall call $S$ the symmetric covariant derivative of $b$ and we will denote it by $\nabla_{\sym} b$. 
\end{definition}

The definition above is the natural counterpart, tailored for vector fields, of the notion of Hessian on $\RCD(K,\infty)$ metric measure spaces (see \cite[Definition 3.3.1]{Gigli14}), which is based in turn on the weak definition of Hessian proposed by Bakry in \cite{Bakry97} in the framework of $\Gamma$-calculus (see also \cite{Sturm14}).

It is easy to verify via the usual calculus rules that, on smooth Riemannian manifolds, smooth vector fields with compact support belong to $H^{1,2}_{C,s}(TX)$ and that the tensor $S$ in \autoref{def:Sobvect} is the symmetric part of the covariant derivative.

\medskip


Following \cite{AmbrosioTrevisan14} we introduce the natural notion of \emph{flow} in this framework.

\begin{definition}[Regular Lagrangian flow]
	\label{def:Regularlagrangianflow}
	We say that $\XX:[0,T]\times X\rightarrow X$ is a \textit{Regular Lagrangian flow} of $b\in L^1([0,T];L^2(TX))$ if the following conditions hold true:
	\begin{itemize}
		\item [(1)] $\XX(0,x)=x$ and $X(\cdot,x)\in C([0,T];X)$ for every $x\in X$;
		\item [(2)] there exists $L\ge0$, called \textit{compressibility constant}, such that
		\begin{equation}\label{eq:compressibility}
			(\XX(t,\cdot))_{*} \meas\leq L\meas,\qquad\text{for every $t\in [0,T]$}\, ;
		\end{equation}
		\item [(3)] for every $f\in \Lip(X,\dist)$, for $\meas$-a.e. $x\in X$ the map $t\mapsto f(\XX(t,x))$ is absolutely continuous and
		\begin{equation}\label{eq: RLF condition 3}
			\frac{\di}{\di t} f(\XX(t,x))= b_t\cdot \nabla f(\XX(t,x)) \quad \quad \text{for a.e.}\ t\in (0,T)\, .	
		\end{equation}
	\end{itemize}
\end{definition}

It has been proven in \cite{AmbrosioTrevisan14} that any bounded vector field $b\in H^{1,2}_{C,s}(TX)$ with bounded divergence admits a unique Regular Lagrangian flow. This means that, if $\XX^1$ and $\XX^2$ are Lagrangian flows associated to $b$ then $\XX^1(t,x)=\XX^2(t,x)$ for any $t\in [0,T]$, for $\meas$-a.e. $x\in X$.
\medskip

Given $s\in [0,T]$ we can define $\XX(s,t, x)$, for $t\in [s,T]$, as the Lagrangian flow of $b$ starting at time $t=s$ from the point $x\in X$. Note that $\XX(0,t,x)=\XX(t,x)$. Exploiting the uniqueness of Lagrangian flows of Sobolev vector fields one can easily check that, for any $0\le s<T$, for $\meas$-a.e. $x\in X$ it holds
\begin{equation}\label{eq:semigroup property}
	\XX(s,t,\XX(s,x))=\XX(t,x),
	\quad \text{for any $t\in [s,T]$}\, .
\end{equation}
It is worth remarking that the assumption $\div b\in L^{\infty}([0,T]\times X)$ allows us to sharpen \eqref{eq:compressibility} into
\begin{equation}\label{eq:divergence bound}
	e^{-t\norm{\div b}_{L^{\infty}}}\meas \le (\XX(t,\cdot))_{*}\meas \le e^{t\norm{\div b}_{L^{\infty}}}\meas,
	\quad\text{for any $t\in [0,T]$},
\end{equation}
as proven in \cite[Theorem 4.6]{AmbrosioTrevisan14}.
\medskip

In order to ease the notation we are going to write $\XX_t(x)$/$\XX_{s,t}(x)$ in place of $\XX(t,x)$ and $\XX(s,t,x)$. We shall also abbreviate Regular Lagrangian flow to RLF sometimes.
\medskip

Readers more interested in Geometric Analysis over smooth Riemannian manifolds are encouraged to assume that $(X,\dist,\meas)$ is a smooth Riemannian manifold equipped with the Riemannian distance and the Riemannian volume measure, and that $b$ is a smooth vector field. Under these assumptions Regular Lagrangian flows are classical flows.
In this case, the interest of the results that we are going to present stands in their quantitative dependence on $\norm{\nabla_{\sym} b}_{L^2}$, $\norm{\div b}_{L^\infty}$ and $t\in [0,T]$.

\subsection*{Regularity of Lagrangian flows}

As we already pointed out, starting from \cite{CheegerColding96}, flows of vector fields with $L^2$ integrability bounds on their derivatives have played a fundamental role in the Geometric Analysis of spaces with lower Ricci curvature bounds. This is basically due the fact that, despite the smoothness of the objects involved, Bochner's inequality naturally guarantees (only) quantitative $L^2$ Hessian bounds on (harmonic) functions in this framework. Thus, when seeking for \emph{stable} estimates, one is forced to develop some tools tailored for integral bounds, see \cite{ColdingNaber12,KapovitchWilking11,KapovitchLi18}.

From another perspective, flows of vector fields with Sobolev regularity on $\setR^n$ were also considered, starting from the seminal \cite{lions}. This field quickly developed, with strong motivations coming mainly from nonlinear problems in Fluid Mechanics and Kinetic Theory.

\medskip

	The regularity theory for flows of Sobolev velocity fields in the Euclidean setting has been pioneered by Crippa and De Lellis in \cite{CrippaDeLellis08}. They proved that, given a Sobolev velocity field $b:\setR^n\to\setR^n$ with bounded divergence, for any $\eps>0$ there exists a Borel set $E_{\eps}$ such that $\haus^n(B_R(0)\setminus E_{\eps})\le \eps$ and 
	\begin{equation}\label{eq:luslip}
		\abs{\XX(t,x)-\XX(t,y)}\le C(T,\eps,\norm{\nabla b}_{L^1(L^2)})\abs{x-y}\, ,\quad\text{for any $x,y\in E_{\eps}$ and $0\le t\le T$}\, .
	\end{equation} 
	This \emph{Lusin-Lipschitz regularity} estimate is weaker than the classical 
	\begin{equation}\label{eq:lipintime}
		\Lip(\XX_t)\le e^{t\Lip(b) }\, , \quad\text{for any } t\ge 0\, ,
	\end{equation}
    holding for the flow of Lipschitz velocity fields.

	\medskip

	In \cite{BrueSemolaahlfors,BrueSemola18}, the first and third authors have proven some versions of \eqref{eq:luslip} in the non-smooth non flat setting of $\RCD(K,N)$ spaces (see \cite[Theorem 2.20]{BrueSemola18}) and used them to show deep structural results for these spaces.
	These estimates, however, despite their strength and usefulness, did not have the expected behaviour with respect to the time variable, making difficult the application of the result, to some extent. 
	More precisely, the issue is that the constant $C$ appearing in the counterparts of \eqref{eq:luslip} in \cite[Theorem 2.20]{BrueSemola18} lacked the expected behaviour with respect to time. Nevertheless, in view of \eqref{eq:lipintime}, it would be desirable to prove estimates like \eqref{eq:luslip} with constants $C$ of the form
	\begin{equation}
		C=1 + t\, C(\eps, T, \norm{\nabla b}_{L^1(L^2)})\, .
	\end{equation} 
	
This is precisely the main goal of this paper.
We recover the natural rate with respect to time in the regularity estimates for RLFs of Sobolev vector fields on $\RCD$ spaces. This will be crucial for some forthcoming developments of the theory \cite{CaputoGigliPasqualetto21} and it is achieved by combining the techniques of \cite{BrueSemola18} and \cite{Deng20}.
\medskip

We will restrict our investigation to \textit{noncollapsed} $\RCD(K,N)$ spaces (see \cite{DePhilippisGigli17,Kitabeppu18} after \cite{CheegerColding97}), i.e. metric measure spaces $(X,\dist,\haus^N)$ satisfying the $\RCD(K,N)$ condition when equipped with the $N$-dimensional Hausdorff measure $\haus^N$, for some $N\in \setN$.\\ 
The reason why we restrict to noncollapsed structures is that they enjoy stronger structural results which allow us to compare the distance functions and Green functions at infinitesimal scales, see \autoref{sec:stabGreen}. Let us recall that Alexandrov spaces and non collapsed Ricci limits are noncollapsed $\RCD$ spaces.

%

\medskip

Before stating the main result we need to introduce a notion of lower/upper approximate slope.

\begin{definition}[lower/upper approximate slope]\label{def:lowuppslope}
	Let $F:X\to X$ be a Borel map. We say that $x\in X$ is a regular point for $F$ if there exists a measurable set $E\subset X$ with density $1$ at $x$ such that $x\in E$ and $F\restr_E$ is Lipschitz continuous. For any regular point $x\in X$ we set
	\begin{equation*}
		\ap_-\abs{DF}(x):=\liminf_{y\in E,\ y\to x}\frac{\dist(F(x),F(y))}{\dist(x,y)}
		\quad
		\text{and}
		\quad
		\ap_+ \abs{DF}(x):=\limsup_{y\in E,\ y\to x}\frac{\dist(F(x),F(y))}{\dist(x,y)}\, .
	\end{equation*}
	We call, respectively, \emph{lower/upper approximate slope} of $F$ at $x\in X$ the nonnegative number $\ap_-\abs{DF}(x)$/$\ap_+\abs{DF}(x)$.	
\end{definition}

\begin{remark}\label{rm:wellposedness}
	Relying on the locally doubling property of $\RCD(K,N)$ spaces, one can easily check that \autoref{def:lowuppslope} does not depend on the particular choice of the set $E\ni x$ with density $1$ at $x$. 
\end{remark}

\begin{remark}\label{rm:smoothcase}
	When $(X,\dist)$ is a smooth Riemannian manifold with the distance induced by the Riemannian metric and $F:X\to X$ is differentiable at $x$, then the upper and lower slopes of $F$ at $x$ correspond, respectively, to the operator norm of $\di F(x)$ and to 
	\begin{equation*}
		\inf_{v\in T_xX\, , \, v\neq 0}\frac{\norm{\di F(x)v}_{F(x)}}{\norm{v}_x}\, .
	\end{equation*}   
\end{remark}

We briefly recall that a point $x\in X$ is said to be regular if the density
   	\begin{equation}\label{density}
   	\theta(x):=	\lim_{r\to 0}\frac{\haus^N(B_r(x))}{\omega_N r^N}\, ,
   	\end{equation}
which exists at any point and in general belongs to $(0,1]$,  satisfies $\theta(x)=1$. By volume convergence and volume rigidity, see \cite[Corollary 1.7]{DePhilippisGigli17} and \cite{CheegerColding97},  this amounts to say that the tangent cone at $x\in X$ is unique and Euclidean of dimension $N$.

Below we state the main result of this note.

\begin{theorem}\label{th:main}
	Let us fix $N\in \setN$, $K\in \setR$ and $T,R>0$.
	Let $(X,\dist,\haus^N)$ be an $\RCD(K,N)$ m.m.s. and $p\in X$ be fixed. For any  $b\in L^2([0,T]; H^{1,2}_{C,s}(TX))$ supported on $B_R(p)$ with $b, \div b\in L^\infty$, there exists a unique Regular Lagrangian flow $\XX_{s,t}$ satisfying the following property. For any $0\le s<T$, for $\haus^N$-a.e. $x\in B_R(p)$ we have that $\XX_{s,t}(x)\in X$ is a regular point and 
	\begin{align}\label{eq:main}
		e^{-2\int_s^tg_r(\XX_{s,r}(x))\di r}\le \ap_-&\abs{D\XX_{s,t}}(x)\\
		&\le \ap_+\abs{D\XX_{s,t}}(x) 
		\le e^{2\int_s^tg_r(\XX_{s,r}(x))\di r}\, , \nonumber
	\end{align}
	for any $t\in [s,T]$, where $g$ is a nonnegative function satisfying
	\begin{equation*}
		\int_0^T\norm{g_r}_{L^2}\di r \le C(B_R(p),K,N)\left\lbrace\norm{\nabla_{\sym} b}_{L^2}+T\norm{\div b}_{L^{\infty}}\right\rbrace\, .
	\end{equation*}
	
	Moreover, when $b$ does not depend on time, there exists a nonnegative function $h\in L^2(X,\haus^N)$ such that
	\begin{equation*}
		\norm{h}_{L^2}\le C(B_R(p),K,N)\left\lbrace\norm{\nabla_{\sym} b}_{L^2}+\norm{\div b}_{L^{\infty}}\right\rbrace
	\end{equation*}
	and, for $\haus^N$-a.e. $x\in B_R(p)$,
	\begin{equation}\label{eq:mainresultautonomous}
		e^{-th(x)}\le \ap_-\abs{D\XX_t}(x) \le \ap_+\abs{D\XX_t}(x) \le e^{th(x)}
		\quad \text{for any $t\in [0,T]$}\, .
	\end{equation}
\end{theorem}

Notice that both the left and right hand side of \eqref{eq:mainresultautonomous} approach $1$ linearly as $t\to 0$, therefore providing a counterpart of \eqref{eq:lipintime} over noncollapsed $\RCD$ spaces and under Sobolev regularity assumptions on the vector field.\\ 
Let us stress that the pointwise nature (instead of almost-everywhere) w.r.t. time of the estimates is a subtle point, and will require indeed some nontrivial arguments.

\medskip

 Starting from \autoref{th:main} and employing again some of the techniques introduced in \cite{ColdingNaber12,KapovitchWilking11}, it is possible to obtain a global regularity estimate, which improves upon those obtained in \cite{BrueSemola18}, since it is H\"older continuous with respect to time.

	\begin{theorem}\label{globallipthm}
		Fix $N\in \setN$, $K\in \setR$ and $H,D, T,R>0$.
		Let $(X,\dist,\haus^N)$ be an $\RCD(K,N)$ m.m.s. and let $p\in X$ be fixed. Let  $b\in L^2([0,T]; H^{1,2}_{C,s}(TX))$ be supported on $B_R(p)$ with $\norm{b}_{L^\infty} + \norm{\div b}_{L^\infty}< D$ and $\norm{\nabla_{\sym}b}_{L^2}<H$. Then, for any $\eps > 0$, there exist $S \subseteq B_{R}(p)$ and $\omega_0(K, N, B_{R}(p), H, D, T, \eps)$, $\alpha(N)$, $C_0(K, N, B_{R}(p), H, D, T, \eps) > 0$ so that
		\begin{equation}
			\haus^N(B_R(p)\setminus S)<\eps\, ,
		\end{equation}
		and for any $x, y \in S$ and any $0 \leq t_1 < t_2 \leq T$ with $t_2 - t_1 \leq \omega_0$, it holds
		\begin{equation}
			1-C_0(t_2-t_1)^{\alpha}\le \frac{\dist(\XX_{t_2}(x), \XX_{t_2}(y))}{\dist(\XX_{t_1}(x), \XX_{t_1}(y))} \leq 1+C_0(t_2-t_1)^{\alpha}\, .
		\end{equation}
	Here $\XX$ denotes the regular Lagrangian flow of $b$.
	\end{theorem}
    
\medskip

To conclude this introductory section, let us comment again on the main new points of the present note. In the setting of smooth Riemannian manifolds with lower Ricci curvature bounds, the previous contributions closest to this topic are the estimates in \cite{KapovitchWilking11,KapovitchLi18}. Therein, following a common pattern within this field, quantitative regularity estimates were obtained via bootstrap along scales starting from qualitative regularity estimates at small scales, that are guaranteed in turn by smoothness. 

Working in the framework of $\RCD$ spaces, there is the necessity to find alternative arguments to start the bootstrap arguments, since neither smoothness is available, nor approximation with smooth objects is possible.
Here we overcome these difficulties combining in a new way the ideas of \cite{Deng20} to handle the time-like behaviour with those in \cite{BrueSemolaahlfors,BrueSemola18} to handle the spatial behaviour of Regular Lagrangian flows.

\subsection*{Plan of the paper}
The remainder of the paper is organised as follows. In \autoref{sec:stabGreen}, which is of independent interest, we deal with asymptotic estimates and converge of Green functions on $\RCD$ spaces. Then \autoref{sec:regflows} collects some material about regularity of Lagrangian flows over $\RCD$ spaces, formulated in terms of Green functions. The material is mainly taken from \cite{BrueSemola18}. In \autoref{sec:almostsurelyregular} we prove that trajectories of Regular Lagrangian flows pass only through regular points starting from almost every point. The last two sections are dedicated to the proofs of \autoref{th:main} and \autoref{globallipthm}, respectively.

\subsection*{Acknowledgements}
The first author is supported by the Giorgio and Elena Petronio
Fellowship at the Institute for Advanced Study.

The third author is supported by the European Research Council (ERC), under the European’s Union Horizon 2020 research and innovation programme, via the ERC Starting Grant “CURVATURE”, grant agreement No. 802689.

Most of this work was developed while the first and third authors were PhD students
at Scuola Normale Superiore. They wish to express their gratitude to this institution for
the excellent working conditions and the stimulating atmosphere.

The authors are grateful to Nicola Gigli for suggesting to them the possibility to sharpen the estimates in \cite{BrueSemola18,BrueSemolaahlfors} and for several stimulating conversations. They are grateful to Andrea Mondino and Enrico Pasqualetto for carefully reading a preliminary version of the note. The second author also thanks Vitali Kapovitch for many helpful discussions.

\section{Stability of Green functions}\label{sec:stabGreen}
The Green function of the Laplacian is a very classical object that, since its introduction in 1830, has been widely used in the study of linear PDEs 
and in geometric analysis. Let us just mention \cite{Colding12,Ding02} for some recent instances close to the topics of the present note.\\
Our interest for this tool comes from the regularity theory for non-smooth flows developed in \cite{BrueSemolaahlfors,BrueSemola18}, where the inverse of the Green function has been used as a replacement of the distance function to measure regularity. Green functions have two remarkable properties that make them more suitable than distance functions for this analysis: they solve equations and they are regular.

\medskip

Given an $\RCD(K,N)$ m.m.s. $(X,\dist,\meas)$ and $\lambda\ge 0$ we define the $\lambda$-Green function by 
\begin{equation}\label{eq:lambdaGreenfunction}
	G^{\lambda}_x(y)=G^{\lambda}(x,y):=\int_0^{\infty} e^{-\lambda t} p_t(x,y)\di t
	\quad \text{for any }x,y\in X,\ \lambda\ge 0\, ,
\end{equation}
where $p_t:X\times X\to[0,+\infty)$ is the so-called \textit{heat kernel} over $(X,\dist,\meas)$. At least formally, $G^{\lambda}$ is a fundamental solution of the operator $-\Delta +\lambda I$. Observe that, in general, the integral in \eqref{eq:lambdaGreenfunction} could be infinite.\\
Due to its particular relevance and in accordance with the classical terminology, when there is no risk of confusion we shall indicate by Green function the $0$-Green function.
\medskip

Let us recall that in \cite{JangLiZhang} 
the classical lower and upper Gaussian heat kernel bounds for manifolds with lower Ricci bounds, originally due to Li and Yau, have been generalised to $\RCD(K,N)$ spaces.
There exist constants $C_1=C_1(K,N)>1$ and $c=c(K,N)\ge0$ such that
\begin{equation}\label{eq:kernelestimate}
	\frac{1}{C_1\meas(B(x,\sqrt{t}))}\exp\left\lbrace -\frac{\dist^2(x,y)}{3t}-ct\right\rbrace\le p_t(x,y)\le \frac{C_1}{\meas(B(x,\sqrt{t}))}\exp\left\lbrace-\frac{\dist^2(x,y)}{5t}+ct \right\rbrace\, ,  
\end{equation}
for any $x,y\in X$ and for any $t>0$. Moreover it holds
\begin{equation}\label{eq:gradientestimatekernel}
	\abs{\nabla p_t(x,\cdot)}(y)\le \frac{C_1}{\sqrt{t}\meas(B(x,\sqrt{t}))}\exp\left\lbrace -\frac{\dist^2(x,y)}{5t}+ct\right\rbrace \quad\text{for $\meas$-a.e. $y\in X$}\, ,
\end{equation}
for any $t>0$ and for any $x\in X$.
We remark that in \eqref{eq:kernelestimate} and \eqref{eq:gradientestimatekernel} above one can take $c=0$ whenever $(X,\dist,\meas)$ is an $\RCD(0,N)$ m.m.s..

\begin{remark}\label{remark:improvement heat kernel estimates}
	A simple scaling argument shows that $C_1$ and $c$ in \eqref{eq:kernelestimate} and \eqref{eq:gradientestimatekernel} satisfy $C_1(K,N)=C_1(N)$ and $c(K,N)=c(N)|K|$. This improves \eqref{eq:kernelestimate} and \eqref{eq:gradientestimatekernel} only when $K$ is negative.

	Indeed, setting $r=1/\sqrt{-K}$ and denoting by $p^{r,x_0}_t(x,y)$ the heat kernel in the $\RCD(-1,N)$ space $\left(X,r^{-1}\dist,\frac{\meas}{\meas(B_r(x_0))}\right)$, it holds
	\begin{equation}\label{eq:heat kernel scaled}
		\meas(B_r(x_0))p_{r^2t}(x,y)=p^{r,x_0}_t(x,y)\, 
		\quad\text{for any }x,y\in X,\ t\ge 0\, .
	\end{equation}
	It is then enough to apply \eqref{eq:kernelestimate} and \eqref{eq:gradientestimatekernel} to $p^{r,x_0}_{t/r^2}(x,y)$ and use the Bishop-Gromov inequality:
	\begin{equation}\label{eq:Bishop-Gromov}
		\frac{\meas(B_R(x))}{v_{K,N}(R)}\le\frac{\meas(B_r(x))}{v_{K,N}(r)}
		\quad \text{for any $0<r<R$ and $x\in X$}\, .
	\end{equation}
    Here $v_{K,N}(r)$ denotes the measure of the ball of radius $r$ on the model space with parameters $K$ and $N$ (see \cite{Villani09}).
\end{remark}

For technical reasons, throughout this section we work under the following

\begin{assumption}\label{Assumption}
	$(X,\dist,\meas)$ is a product between an $\RCD(K,N-3)$ m.m.s. and a Euclidean factor $(\setR^3,\dist_{\mathrm{eucl}},\leb^3)$, for some $4<N<\infty$.
\end{assumption}

Building upon \eqref{eq:kernelestimate} and \eqref{eq:gradientestimatekernel} one can check that, for $\lambda\ge\lambda(K)$, for any $x\in X$, $G^{\lambda}_x, |\nabla G^{\lambda}_x|\in L^1_{\loc}(X,\meas)$ and
$\Delta G_x^{\lambda}=-\delta_x+\lambda G_x^{\lambda}$, see \cite[subsection 2.3]{BrueSemola18} for further explanations.
\medskip

We refer to \cite{AmbrosioHonda17,GigliMondinoSavare15} for the relevant background about convergence of functions and Sobolev spaces along converging sequences of $\RCD(K,N)$ spaces.
\medskip

Below we state the main convergence result for Green functions along converging sequences of $\RCD(K,N)$ spaces and then we specialize it to the case of tangent cones.

\begin{proposition}\label{prop:Green convergence}
	Let $(X,\dist,\meas)$ be an $\RCD(K,N)$ m.m.s. satisfying \autoref{Assumption} and let $r_i\downarrow 0$ be a sequence of radii such that
	\begin{equation*}
		\lim_{i\to \infty} \left(X,r_i^{-1}\dist,\frac{\meas}{\meas(B_{r_i}(x_0))},x_0\right)=(Y,\rho,\mu,y)
		\quad\text{in the pmGH topology}\, .
	\end{equation*}
	Denoting by $G^{\lambda}$ the $\lambda$-Green function in $(X,\dist,\meas)$ and by $G$ the $0$-Green function in $(Y,\rho,\mu,y)$ (see \eqref{eq:lambdaGreenfunction}) one has
	\begin{equation}\label{eq: Green convergence}
		\lim_{i\to \infty} r_i^{-2}\meas(B_{r_i}(x_0))G^{\lambda}(x_i,y_i)\to G(x_{\infty},y_{\infty})\, ,
	\end{equation}
	for $X_i\times X_i\ni (x_i,y_i)\to (x_{\infty},y_{\infty})\in Y\times Y$ and $\lambda\ge c|K|$, where the constant $c$ is the one from \eqref{eq:kernelestimate} and \eqref{eq:gradientestimatekernel}.	
\end{proposition}

\begin{corollary}\label{corollary:Green convergence}
	Let $(X,\dist,\haus^N)$ be a noncollapsed $\RCD(K,N)$ space satisfying \eqref{Assumption}. For $\lambda\ge c|K|$ and $x\in X$ one has
	\begin{equation}
		\lim_{y\to x} \dist(x,y)^{N-2}G^{\lambda}(x,y)=\frac{1}{\theta(x)\omega_N N(N-2)}\, ,
	\end{equation}
	where $\theta\in (0,1]$ is the density of $\haus^N$ at $x$, as defined in \eqref{density}.
\end{corollary}

\begin{remark}
Even though this will be not relevant for our purposes, let us point out that analogous conclusions hold when considering the limiting behaviour of the Green function $G$ on blow-downs (i.e. tangent cones at infinity instead of local tangent cones) of $\RCD(0,N)$ metric measure spaces $(X,\dist,\haus^N)$ with Euclidean volume growth for $N\ge 3$.
\end{remark}

\subsection{Proof of \autoref{prop:Green convergence}}
We recall a convergence result for heat kernels, referring the reader to \cite[Theorem 3.3]{AmbrosioHondaTewodrose17} for its proof.

\begin{lemma}\label{th:heat Kernel convergence}
	Let $\left((X_i,\dist_i,\meas_i,x_i)\right)_i$ be a sequence of $\RCD(K,N)$ m.m.spaces converging in the pmGH topology to $(X_{\infty},\dist_{\infty},\meas_{\infty},x_{\infty})$. Then the heat kernels $p^i$ of $X_i$ satisfy
	\begin{equation}
		\lim_{i\to \infty} p^i_{t_i}(x_i,y_i)=p^{\infty}_t(x,y)\, ,
	\end{equation}
	for any $X_i\times X_i\times (0,\infty)\ni (x_i,y_i,t_i)\to (x,y,t)\in X_{\infty}\times X_{\infty}\times (0,\infty)$, where $p^{\infty}$ denotes the heat kernel in $X_{\infty}$.
\end{lemma}

When $N\ge 3$ and $(X,\dist,\meas)$ is an $N$-metric measure cone with tip $p$ over an $\RCD(N-2,N-1)$ m.m.s. (see \cite{DePhilippisGigli16}), the Green function of the Laplacian, centered at $p$, coincides, up to a multiplicative constant, with the distance function raised to the power $(2-N)$. This is a consequence of separation of variables, see \cite{GigliHan18}. We omit the proof, since it can be obtained as in the case of Ricci limit spaces considered in \cite{Ding02} (see also the previous \cite{ColdingMinicozzi97}, which is the first appearance of this principle to the best of our knowledge, and \cite[Subsection 4.10]{CheegerJiangNaber18} for analogous results and computations).

\begin{lemma}\label{lemma:Green on cones}
	Let $N\ge 3$ and $c>0$ be given. Let $(Y,\rho,c\haus^N)$ be an $\RCD(0,N)$ m.m.s.. 
	If $(Y,\rho)$ is a metric cone with tip $p\in Y$, then there exists a positive Green function of the Laplacian $G$ on $Y$ given by \eqref{eq:lambdaGreenfunction} and
	\begin{equation}
		G(p,x)=\frac{\rho(p,x)^{2-N}}{(N-2)N c\haus^N(B_1(p))}\, ,
		\quad\text{for any $x\neq p$}\, .
	\end{equation}
\end{lemma}

The last lemma shows that, on noncollapsed ambient spaces, $G^{\lambda}(x,y)$ is locally uniformly equivalent to $\dist(x,y)^{2-N}$ on bounded sets, for suitable choices of $\lambda$. It reflects the classical local equivalence between Green's functions and negative powers of the distance on smooth Riemannian manifolds, see for instance \cite{Aubin98}.

\begin{lemma}\label{prop:equivalenceGreendistance}
	Let $(X,\dist,\haus^N)$ an $\RCD(K,N)$ m.m.s. satisfying \autoref{Assumption}.
	Then, for any $\lambda\ge c|K|$, $p\in X$ and $R>0$, there exists a constant $C_1=C_1(B_R(p),K,N,\lambda)>0$ such that
	\begin{equation}\label{z0}
		\frac{C_1^{-1}}{\dist(x,y)^{N-2}}\le G^{\lambda}(x,y)
		\le \frac{C_1}{\dist(x,y)^{N-2}}\, ,
		\quad \text{for any $x,y\in B_R(p)$}\, .
	\end{equation}
\end{lemma}
\begin{proof}
	Arguing as in the proof of \cite[Proposition 2.21]{BrueSemola18}, where the case $\lambda=c\abs{K}$ is considered, relying on \cite{Grygorian06} it is possible to prove that, for any $\lambda\ge c|K|$, $p\in X$ and $R>0$ there exists a constant $C=C(\lambda,B_R(p))>0$ such that
	\begin{equation}\label{eq:compestimators}
		C^{-1}\int_{\dist(x,y)}^{\infty}\frac{r}{\haus^N(B_r(x))}\di r
		\le G^{\lambda}(x,y)\le C\int_{\dist(x,y)}^{\infty}\frac{r}{\haus^N(B_r(x))}\di r\, ,
		\quad \text{for any $x,y\in B_R(p)$}\, .
	\end{equation}
	By the Bishop-Gromov inequality \eqref{eq:Bishop-Gromov} and the noncollapsing assumption it holds
	\begin{equation}\label{z2}
		C^{-1}(K,N) r^N\le \haus^N(B_r(x))\le  C(K,N) r^N\, ,
		\quad\text{for any $x\in B_R(p)$ and $0<r<5R$}\, .
	\end{equation}
On the other hand, \autoref{Assumption} yields 
	\begin{equation}\label{z3}
		\haus^N(B_r(x))\ge 2 r^3
		\quad\text{for any $x\in X$ and $r>0$}\, .
	\end{equation}
	The conclusion follows combining \eqref{eq:compestimators}, \eqref{z2} and \eqref{z3}.
\end{proof}

\medskip

\begin{proof}[Proof \autoref{prop:Green convergence}]
	Using \eqref{eq:heat kernel scaled} we can write
	\begin{equation}\label{eq:changevariable}
	\int_0^{\infty} e^{-\lambda r^2 t}p^{r,x_0}_t(x,y)\di t
	=
	\meas(B_r(x_0))\int_0^{\infty}e^{-\lambda r^2 t}p_{r^2t}(x,y)\di t
	=
	r^{-2}\meas(B_r(x_0))G^{\lambda}(x,y),
	\end{equation}
	for any $x,y\in X$. Hence, \eqref{eq: Green convergence} will follow from \eqref{eq:changevariable} applying the dominated convergence theorem, thanks to \autoref{th:heat Kernel convergence} and the bound
	\begin{equation}\label{z1}
	e^{-\lambda r_i^2 t}p_t^{r_i,x_0}(x_i,y_i)\le
	\begin{dcases}
	C(N,K)C_2 t^{-3/2}e^{-\frac{\rho(x_{\infty},y_{\infty})^2}{10 t}}& \text{for }t\ge 1\, ;\\
	C(N,K)C_2t^{-N/2}e^{-\frac{\rho(x_{\infty},y_{\infty})^2}{10 t}}&
	\text{for } t<1\, ,
	\end{dcases}
	\end{equation}
	which is valid for any $i\in \setN$ big enough.\\	
	Let us check \eqref{z1}. Using the heat kernel estimate \eqref{eq:kernelestimate} and \autoref{remark:improvement heat kernel estimates} one has
	\begin{align*}
	e^{-\lambda r_i^2 t}p_t^{r_i,x_0}(x_i,y_i) & \le e^{-r_i^2t(\lambda-c|K|)}C_1 \frac{\meas(B_{r_i}(x_0))}{\meas(B_{r_i\sqrt{t}}(x_0))}e^{-\left(\frac{\dist(x_i,y_i)}{r_i}\right)^2\frac{1}{5t}}\, .
	\end{align*}
	This estimate, along with the assumption $\lambda\ge c|K|$ and $\lim_{i\to \infty}r_i^{-1}\dist(x_i,y_i)=\rho(x_{\infty},y_{\infty})$, gives	
	\begin{align*}
	e^{-\lambda r_i^2 t}p_t^{r_i,x_0}(x,y)  \le C_1 \frac{\meas(B_{r_i}(x_0))}{\meas(B_{r_i\sqrt{t}}(x_0))}e^{-\frac{\rho(x_{\infty},y_{\infty})^2}{10t}}\, ,
	\quad
	\text{for any $i\in\setN$ big enough}\, .
	\end{align*}
	The inequality \eqref{z1} follows bounding $\frac{\meas(B_{r_i}(x_0))}{\meas(B_{r_i\sqrt{t}}(x_0))}$ with
	\begin{equation}\label{eq:technicalcondition}
		\sup_{x\in X,\ r\in (0,1)}\frac{\meas(B_{r}(x))}{\meas(B_{rM}(x))}\le \frac{C(R,K)}{M^3}	\, ,
		\quad \text{for any $M\ge 1$, $r\le  R$}\, ,
	\end{equation}
	for $t\ge 1$, and with the Bishop-Gromov inequality \eqref{eq:Bishop-Gromov} for $t< 1$.
	The estimate \eqref{eq:technicalcondition} can be checked exploiting \autoref{Assumption} and again the Bishop-Gromov inequality \eqref{eq:Bishop-Gromov}.
\end{proof}

\subsection{Proof of \autoref{corollary:Green convergence}}
	It is enough to prove that for any $y_i\to x$ there exists a subsequence $(i_k)$ such that
	\begin{equation}
		\lim_{k\to \infty} \dist(x,y_{i_k})^{N-2}G^{\lambda}(x,y_{i_k})=\frac{1}{\theta(x)\omega_N N(N-2)}\, .
	\end{equation}
	To this end, we set $r_i:=\dist(x,y_i)$ and, up to extracting a subsequence that we do not relabel, we assume that
	\begin{equation*}
		\left(X,r_i^{-1}\dist, \haus^N/ \haus^N(B_{r_i}(x_0)),x_0\right)\to (Y,\rho,\haus^N/\haus^N(B_1(y)),y)\, ,
		\quad\text{in the pmGH topology}
	\end{equation*}
	and that $X_i\ni y_i\to y_{\infty}\in Y$.
	
	Using \autoref{prop:Green convergence} we have
	\begin{equation*}
		\lim_{i\to \infty}\dist(x,y_i)^{N-2}G^{\lambda}(x,y_i)
		=\lim_{i\to \infty}  \frac{r_i^N}{\haus^N(B_{r_i}(x))}       r_i^{-2}\haus^N(B_{r_i}(x))G^{\lambda}(x,y_i)
		=\frac{G^Y(y,y_{\infty})}{\omega_N \theta(x)}\, .
	\end{equation*}
	To conclude, we can apply \autoref{lemma:Green on cones} with $c=1/\haus^N(B_1(y))$ and observing that $\rho(y,y_{\infty})=1$, due to the choice of the rescaling.

 \section{Regularity for Lagrangian Flows via Green functions}\label{sec:regflows}

   In this section we collect some known regularity results for flows of Sobolev velocity fields taken from \cite{BrueSemola18,BrueSemolaahlfors}.
   \medskip

   We fix a noncollapsed $\RCD(K,N)$ metric measure space $(X,\dist,\haus^N)$ satisfying \autoref{Assumption}, a point $p\in X$ and $R>0$. Then we consider a vector field $b\in L^1([0,T];H^{1,2}_{C,s}(TX))$ with $\supp b\subset B_R(p)$ uniformly in time, and we set
   \begin{equation}
   \norm{b}_{L^\infty} + \norm{\div b}_{L^\infty} =: D <\infty\, .
   \end{equation}   
   Let us also set
   \begin{equation*}
   	\dist_{G^{\lambda}}(x,y):=\frac{1}{G^{\lambda}(x,y)}\, .
   \end{equation*}

   \begin{proposition}[Estimate for the trajectories]\label{prop:esttrajbar}
   	Let $(X,\dist,\haus^N)$ and $b$ be as above, let $\XX$ be a Regular Lagrangian flow of $b$ and $\lambda>c|K|$. Then, for any $0\le s<T$ and $\haus^N\times\haus^N$-a.e. $(x,y)\in B_R(p)\times B_R(p)$, it holds
   \begin{equation}\label{eq:est1bar}
   e^{-\int_s^t(g_r(\XX_{s,r}(x))+g_r(\XX_{s,r}(y)))\di r}\le 
   \frac{\dist_{G^{\lambda}}(\XX_{s,t}(x),\XX_{s,t}(y))}{ \dist_{G^{\lambda}}(x,y)}
   \le e^{\int_s^t(g_r(\XX_{s,r}(x))+g_r(\XX_{s,r}(y)))\di r}\, ,
   \end{equation}   
   for any $t\in [s,T]$. Here $g$ is a nonnegative function such that
   \begin{equation}\label{eq:boundg}
   	\int_0^T\norm{g_r}_{L^2}\di r \le C(B_R(p),\lambda,K,N)\left\lbrace\norm{\nabla_{\sym} b}_{L^1(L^2)}+T\norm{\div b}_{L^{\infty}}\right\rbrace\, .
   \end{equation}
   \end{proposition}

   The main ingredient for the proof of \autoref{prop:esttrajbar} is the following maximal estimate for time independent velocity fields. We refer the reader to \cite[Proposition 2.27]{BrueSemola18} for its proof.

	\begin{proposition}[Maximal estimate, vector-valued version]\label{prop: barG-maximal estimate}
    Let $(X,\dist,\haus^N)$ be a noncollapsed $\RCD(K,N)$ m.m.s., $b\in H^{1,2}_{C,s}(TX)$ with $\div b\in L^2(X)$ and $\lambda>c|K|$ as above.
	Then, there exists a positive function $g\in L^2(B_R(p),\haus^N)$ such that
	\begin{equation}\label{z14}
	\abs{ b\cdot\nabla G^{\lambda}_x(y)+ b\cdot\nabla G^{\lambda}_y(x)}\le G^{\lambda}(x,y)(g(x)+g(y))\, ,
	\end{equation}
    for $\haus^N\times\haus^N$-a.e. $(x,y)\in B_R(p)\times B_R(p)$, and
	\begin{equation}\label{eq: z}
	\norm{g}_{L^2(B_R(p))}\le C_V \norm{\nabla_{\sym} b}_{L^2}+\norm{\div b}_{L^2}\, ,
	\end{equation}
	where $C_V=C_V(B_R(p),\lambda, K, N)>0$.	
\end{proposition}

\begin{proof}[Proof of \autoref{prop:esttrajbar}]
	It is enough to show that, for any $s\in[0,T)$ and for $\haus^N\times \haus^N$-a.e. $(x,y)\in B_R(p)\times B_R(p)$, it holds
	\begin{equation}\label{eq:GreenLusinLipschitz}
	 e^{-\int_s^t(g_r(\XX_r(x))+g_r(\XX_r(y)))\di r}\le \frac{\dist_{G^{\lambda}}(\XX_t(x),\XX_t(y))}{\dist_{G^{\lambda}}(\XX_s(x),\XX_s(y))}\le 
	 e^{\int_s^t(g_r(\XX_r(x))+g_r(\XX_r(y)))\di r}\, ,
	\end{equation}
	for any $t\in [s,T]$.\\
	Indeed, exploiting \eqref{eq:semigroup property} we can rewrite \eqref{eq:GreenLusinLipschitz} as follows: for any $0\le s<T$ and for $\haus^N\times \haus^N$-a.e. $(x,y)\in B_R(p)\times B_R(p)$ it holds
	\begin{align*}
	\exp&\left\lbrace
	-\int_s^t\left(g_r(\XX_{s,r}(\XX_s(x)))+g_r(\XX_{s,r}(\XX_s(y)))\right)\di r \right\rbrace\\
	&\le \frac{\dist_{G^{\lambda}}(\XX_{s,t}(\XX_s(x)),\XX_{s,t}(\XX_s(y)))}{\dist_{G^{\lambda}}(\XX_s(x),\XX_s(y))}\\
	&\le 
	\exp\left\lbrace
	\int_s^t\left(g_r(\XX_{s,r}(\XX_s(x)))+g_r(\XX_{s,s}(\XX_s(y)))\right)\di r\right\rbrace\, ,
	\end{align*}
	for any $t\in [s,T]$. Then we can use \eqref{eq:divergence bound} to change variable and get \eqref{eq:est1bar}.
	
	Let us prove \eqref{eq:GreenLusinLipschitz}. By \cite[Corollary A.3]{BrueSemola18} and \autoref{prop: barG-maximal estimate} we get that
	\begin{equation}\label{eq:estdiff}
	\abs{ \frac{\di}{\di r}  G^{\lambda}(\XX_r(x),\XX_r(y))} \leq G^{\lambda}(\XX_r(x),\XX_r(y))\left\lbrace g_r(
	\XX_r(x))+g_r(\XX_r(y))\right\rbrace \, ,
	\end{equation}
	for $\leb^1$-a.e. $r\in (0,T)$ and
	for $\haus^N\times \haus^N$-a.e. $(x,y)\in B_R(p)\times B_R(p)$.\\
	Integrating \eqref{eq:estdiff} with respect to the time variable and recalling that $\dist_{G^{\lambda}}:=1/G^{\lambda}$, we get \eqref{eq:GreenLusinLipschitz}.
\end{proof}

\subsection{Lusin-Lipschitz estimate for Lagrangian flows}
Exploiting the local equivalence proved in \autoref{prop:equivalenceGreendistance} we can now turn the Lusin-Lipschitz estimate in terms of $G^{\lambda}$ into a classical Lusin-Lipschitz estimate with respect to the distance $\dist$. We refer the reader to \cite{BrueSemolaahlfors} for an analogous statement in the case of compact Ahlfors regular $\RCD(K,N)$ spaces.


\begin{proposition}\label{prop:LusinLipschitzFlows}
    Let $(X,\dist,\haus^N)$ be an $\RCD(K,N)$ m.m.s. satisfying \autoref{Assumption}. Let us fix a point $p\in X$ and $R>0$. Then, let us consider a vector field $b\in L^1([0,T];H^{1,2}_{C,s}(TX))$ with $\supp b\subset B_R(p)$ uniformly in time, and set $\norm{b}_{L^\infty} + \norm{\div b}_{L^\infty} =: D <\infty.$\\	 
	Then, for any $s\in [0,T]$, there exist a nonnegative function $g_s':B_R(p)\to [0,\infty]$ and a positive constant $C_3=C_3(K,N,B_R(p))$ such that, for any $x,y\in B_R(p)$, it holds
	\begin{equation}\label{eq:LusinLipschitzflows}
		\frac{\dist(\XX_{s,t}(x),\XX_{s,t}(y))}{\dist(x,y)}
		\le C_3 e^{\left(g_s'(x) + g_s'(y)\right)}\, ,
		\quad \text{for any $0\le s\le t\le T$}
	\end{equation}
     and 
	\begin{equation*}
	\norm{g_s'}_{L^2} \le C(B_R(p), D, K,N)\left\lbrace\norm{\nabla_{\sym} b}_{L^1(L^2)}+T\norm{\div b}_{L^{\infty}}\right\rbrace\, .
	\end{equation*}
\end{proposition}
\begin{proof}
	As a consequence of \autoref{prop:esttrajbar} and \eqref{prop:equivalenceGreendistance},  for any $0\le s<T$, for $\haus^N\times\haus^N$-a.e. $(x,y)\in B_R(p)\times B_R(p)$ it holds
	\begin{equation*}
	\frac{\dist(\XX_{s,t}(x),\XX_{s,t}(y))}{\dist(x,y)}
	\le C_1^2 \exp\left\lbrace \int_s^t g_r(\XX_{s,r}(x))\di r+\int_s^t  g_r(\XX_{s,r}(y))\di r \right\rbrace\, ,
	\end{equation*}
	for any $t\in [s,T]$.	
The sought conclusion follows applying a local version of \autoref{lemma:technical} below choosing $h(x)=h_s(x):=\int_s^{T}g_r(\XX_{s,r}(x))\di r$.
\end{proof}

\begin{lemma}\label{lemma:technical}
	Let $(X,\dist,\meas)$ be a locally doubling m.m.s., let $F:X\to X$ be a measurable function and $h\in L^2(X,\meas)$. If
	\begin{equation*}
		\dist(F(x),F(y))\le Ce^{\left(h(x)+h(y)\right)}\dist(x,y)
		\quad\text{for $\meas\times \meas$-a.e. $(x,y)\in X\times X$,}
	\end{equation*}
	 then there exists a function $h':X\to [0,+\infty]$ such that
	\begin{equation*}
		\dist(F(x),F(y))\le C'e^{h'(x)+h'(y)}\dist(x,y)
		\quad \text{for any } x,y\in X
		\quad \text{and}\quad \norm{h'}_{L^2}\le C'\norm{h}_{L^2}\, ,
	\end{equation*}
	where $C'$ depends only on $C$ and the doubling constant of $\meas$.
\end{lemma}
\begin{proof}
	We do not give here a complete proof of this statement. Let us just point out that it can be obtained arguing as in the proof of \cite[Theorem 2.20]{BrueSemola18} (see also \cite{CrippaDeLellis08} for the original argument in Euclidean spaces).
\end{proof}

The lemma above applies in particular to any $\RCD(K,N)$ metric measure space $(X,\dist,\meas)$, since the local doubling property follows from the Bishop-Gromov inequality.

%
%

\section{Trajectories almost surely pass through regular points}\label{sec:almostsurelyregular}

In this section we will show that the trajectory of the regular Lagrangian flow $\XX_t$ of a time dependent vector field $b \in L^2([0,T]; H^{1,2}_{C,s}(TX))$ with bounded divergence (and so in particular autonomous vector fields satisfying proper covariant derivative and divergence bounds) passes only through regular points starting from $\haus^N$-a.e. $x$. 
\medskip

The techniques we will use are similar to those in \cite{KapovitchWilking11, ColdingNaber12, KapovitchLi18} (see also \cite{Deng20} in the $\RCD$ setting). In essence, we will bootstrap the existence of the nonoptimal Lipschitz bounds between trajectories arising from \autoref{prop:esttrajbar} and \autoref{prop:equivalenceGreendistance} to obtain uniform H\"older estimates on the volume of arbitrarily small balls (depending on the trajectory but independent of the radius of the balls) along almost all trajectories. This will show that the density $\theta(\XX_t(x))$ changes continuously w.r.t. $t$, for $\haus^N$-a.e. $x$. In view of the fact that for $\haus^N$-a.e. $x$, for almost every $t \in [0,T]$, $\XX_t(x)$ is regular (equivalently, $\theta(\XX_t(x))=1$ for a.e. $t\in[0,T]$) and using again volume rigidity \cite[Corollary 1.7]{DePhilippisGigli17}, this is enough to show that almost all trajectories pass through only regular points (equivalently, $\theta(\XX_t(x))=1$ for every $t\in[0,T]$).
\medskip

After dealing with the general case, we are going to present a technically simpler argument tailored for the framework of spaces without boundary and based on \cite{Aizenman}.

\subsection{The general case}
For the rest of the section, we consider an $\RCD(K,N)$ m.m.s. $(X, \dist, \haus^N)$ satisfying \autoref{Assumption}. We fix some $p \in X$ and $R, T, D, H > 0$. For simplicity, we will  consider the Green function $G^{\lambda}$ where $\lambda = c|K|$. We also fix a time dependent bounded vector field $b \in L^2([0,T]; H^{1,2}_{C,s}(TX))$ with $\supp(b_t) \subset B_R(p)$, $\norm{b}_{L^{\infty}}+\norm{\div b}_{L^{\infty}} \leq D$, and $\int_{0}^{T} \norm{|\nabla_{\sym}b_t|}_{L^2}^2 \di t \leq H$.\\ 
We will continue to use the notations $\XX_t$ and $\XX_{s,t}$ as before. We fix a representative of $\XX_t$ starting from here and assume that, for all $x \in X$, $\XX_t(x)$ is a Lipschitz curve with Lipschitz constant $D$. 

\medskip

To begin, we fix a collection of constant speed geodesics $\gamma_{x,y}$ from each $x \in X$ to each $y \in X$ so that the map $X \times X \times [0,1] \ni (x,y,t) \mapsto \gamma_{x,y}(t)$ is Borel. This is possible thanks to the Kuratowski and Ryll-Nardzewski measurable selection theorem, see \cite[Remark 2.26]{Deng20} and references therein. 

We will also need the notion of the \emph{distance distortion function} to keep track of the distance between points. The terminology and definition come from \cite{KapovitchWilking11}.\\ 
Given two RLFs $F_t, G_t: X \times [0,T] \to X$ and $t \in [0,T]$, we define $dt^{F, G}_r(t): X \times X \to [0,r]$, the distance distortion function on the scale $r$, by 
\begin{equation}
	dt^{F, G}_r(t)(x,y) := \min\{r, \max\limits_{0 \leq \tau \leq t} |\dist(x,y)-\dist(F_{\tau}(x), G_{\tau}(y))|\}\, .
\end{equation}
We use $dt^{F}_r(t)$ to denote $dt^{F, F}_r(t)$.
\medskip

The following proposition is a slight generalization of \cite[Proposition 3.27]{Deng20}, which is proved using a localization \cite[Proposition 3.23]{Deng20} of the second order differentiation formula shown in \cite[Theorem 5.13]{GigliTamanini18}. 

\begin{proposition}\label{prop:dtcontrol}
	Let $W \in L^1([0,T]; H^{1,2}_{C,s}(TX))$ and $F_t$, $G_t$ be RLFs corresponding to bounded $U, V \in L^1([0,T]; L^2(TX))$ respectively. Let $S_1, S_2$ be Borel subsets of $X$ with finite positive measure. The map $t \mapsto \int_{S_1 \times S_2} dt_{r}^{F, G}(t)(x,y) \di (\haus^N \times \haus^N)(x,y)$ is Lipschitz on $[0,T]$ and satisfies
\begin{align*}
	&\hspace{0.5cm}\frac{\di}{\di t} \int_{S_1 \times S_2} dt^{F,G}_r(t)(x,y) \, \di(\haus^N \times \haus^N)(x,y) \\
	&\leq \int_{\Gamma_r(t)} \big(|U_t - W_t|(F_t(x)) + |V_t - W_t|(G_t(y)) \big)\, \di(\haus^N \times \haus^N)(x,y)\\
	&\hspace{2cm}+ \int_{0}^{1} \int_{\Gamma_r(t)} \dist(F_t(x),G_t(y)) |\nabla_{\sym} W_t|(\gamma_{F_t(x),G_t(y)}(s)) \, \di(\haus^N \times \haus^N)(x,y) \, \di s\, ,
\end{align*}	
for $\mathscr{L}^1$-a.e. $t \in [0,T]$, where $\Gamma_r(t) := \{(x,y) \in S_1 \times S_2 \, : \, dt_r^{F, G}(t)(x,y) < r\}.$
\end{proposition}
We note that the generalization is in two directions, the possibility that $W_t$ is time dependent and in $H^{1,2}_{C,s}(TX)$ instead of $H^{1,2}_{C}(TX)$.\\ 
The proof of \cite[Proposition 3.27]{Deng20} generalizes easily in the former direction. For the latter, we note that by the discussion of \cite[Remark 2.6]{BrueSemolaahlfors}, \cite[Theorem 5.13]{GigliTamanini18} holds as stated for vector fields in $H^{1,2}_{C,s}(TX)$ with $\nabla$ replaced by $\nabla_{sym}$, which is all that is needed. 

The following corollary follows by replacing $U$, $V$ and $W$ with $b$ in \autoref{prop:dtcontrol}. 

\begin{corollary}\label{corollary:dtcontrolcor}
	Let $S_1, S_2$ be Borel subsets of $X$ with finite positive measure. Then the map $t \mapsto \int_{S_1 \times S_2} dt_{r}^{\XX}(t)(x,y) \di (\haus^N \times \haus^N)(x,y)$ is Lipschitz on $[0,T]$ and satisfies
\begin{align*}
	&\hspace{0.5cm}\frac{\di}{\di t} \int_{S_1 \times S_2} d t^{\XX}_r(t)(x,y) \, \di (\haus^N \times \haus^N)(x,y) \\
	&\leq \int_{0}^{1} \int_{\Gamma_r(t)} \dist(\XX_t(x),\XX_t(y)) |\nabla_{\sym} b_t|(\gamma_{\XX_t(x),\XX_t(y)}(s)) \, \di (\haus^N \times \haus^N)(x,y) \, \di s
\end{align*}	
for $\mathscr{L}^1$-a.e. $t \in [0,T]$, where $\Gamma_r(t) := \{(x,y) \in S_1 \times S_2 \, : \, dt_r^{\XX}(t)(x,y) < r\}.$
\end{corollary}

Below we state and prove the main result of this section. 

\begin{theorem}\label{measurecontinuity}
Let $(X,\dist,\haus^N)$ be a noncollapsed $\RCD(K,N)$ m.m.s., $p\in X$ and let $b$, $D$ and $H$ be as above. Then for $\haus^N$-a.e. $x \in B_{R}(p)$, there exist $r_x > 0$ and a modulus of continuity $g_x:[0,\infty)\to[0,\infty)$ such that $g(0)=0$, $g$ is continuous at $0$ and the following holds: 
\begin{equation}\label{eq:ucvolumeratio}
\abs{\frac{\haus^N(B_r(\XX_{t_1}(x)))}{\haus^N(B_r(\XX_{t_2}(x)))}-1}\le g(\abs{t_2-t_1})\, ,\quad\text{for any $0<r<r_x$ and any $0\le t_1,t_2\le T$}\,.
\end{equation}

As a corollary, for $\haus^N$-a.e. $x\in B_R(p)$, $\XX_t(x)$ is a regular point for any $t\in [0,T]$.
\end{theorem}

\begin{proof}
Fix any $\eps > 0$. It suffices to show the claim holds for the elements of some $S \subseteq B_{R}(p)$ with $\haus^N(B_{R}(p) \setminus S) \leq \eps$.
\medskip

Fix $C_1(K,N,B_{R}(p))$ as in \autoref{prop:equivalenceGreendistance} (notice the dependence on $\lambda$ is dropped since we assume $\lambda = c|K|$). Fix some $g \in L^1([0,T]; L^2(B_R(p),\haus^N))$ as in \autoref{prop:esttrajbar} for $b$. Note
\begin{align}\label{eq:gstructuralbound}
\begin{split}
	&\hspace{0.5cm} \int_{0}^{T} \int_{B_R(p)} g_s(\XX_s(x))\, \di \haus^N(x)\, \di s\\
	 &\leq e^{DT} \int_{0}^{T} \int_{B_R(p)} g_s(x) \di \haus^N(x)\, \di s\\
	& \leq e^{DT}\sqrt{\haus^N(B_{R}(p))} \int_{0}^{T} \norm{g_s}_{L^2}\, \di s\\
	& \leq e^{DT}\sqrt{\haus^N(B_{R}(p))}c(B_{R}(p), K, N)\left(\int_{0}^{T} \norm{\nabla_{\sym}b_s}_{L^2}\, \di s+T\norm{\div b}_{L^{\infty}}\right)\\
	&\leq e^{DT}\sqrt{\haus^N(B_{R}(p))}c(B_{R}(p), K, N)(\sqrt{TH}+TD) =: C_2(B_{R}(p), K, N, H, D, T)\, ,
\end{split}
\end{align}
where we used \eqref{eq:divergence bound}, Cauchy-Schwarz inequality, the bound \eqref{eq:boundg} on $\int_{0}^{T} \norm{g_r}_{L^2}\di r$ and the definitions of $D, H$ from the beginning of the section.
\medskip

Let $E_1$ be the set of $x \in B_{R}(p)$ for which \eqref{eq:est1bar} holds for $s=0$ and $\haus^N$-a.e. $y$. By Fubini's theorem, $\haus^N(B_{R}(p) \setminus E_1) = 0$.\\ 
Let $E_2$ be the set of $x \in B_{R}(p)$ for which $\int_{0}^{T} g_r(\XX_r(x)) \di r \leq M_1$, where, by \eqref{eq:gstructuralbound} and Chebyshev's inequality, $M_1(B_{R}(p), K, N, H, D, T, \eps)$ is chosen sufficiently large so that $\haus^N(B_{R}(p) \setminus E_2) \leq \eps/2$. 

For each $t \in [0,T]$, define the maximal function $Mx_{t}$ of $|\nabla_{\sym} b_{t}|$ for $x \in X$ by 
\begin{equation*}
Mx_{t}(x) := \sup\limits_{0 < r \leq 16R} \fint_{B_{r}(x)} |\nabla_{\sym} b_{t}|(z) \di \haus^N(z)\, .
\end{equation*}
By the standard maximal inequality and using that $b_t$ is supported in $B_R(p)$, we have $\norm{Mx_{t}}_{L^2} \leq c(K, N, R)\norm{|\nabla_{\sym}b_{t}|}_{L^2}$. Therefore, using again \eqref{eq:divergence bound},
\begin{align}\label{eq:Mxstructuralbound}
\begin{split}
	\int_{0}^{T} \int_{B_R(p)} Mx^2_s(\XX_s(x))\, \di \haus^N(x)\, \di s
	 &\leq e^{DT} \int_{0}^{T} \int_{B_R(p)} Mx^2_s(x) \di \haus^N(x)\, \di s\\
	& \leq e^{DT}\int_{0}^{T} c^2\norm{|\nabla_{\sym}b_{t}|}_{L^2}^2\, \di s\\
	& \leq e^{DT}c^2H =: C_3(B_{R}(p), K, N, H, D, T)\, .
\end{split}
\end{align}
Let $E_3$ to be the set of $x \in B_{R}(p)$ for which $\int_{0}^{T} Mx^2_{s}(\XX_{s}(x)) \di s \leq M_2$, where, by \eqref{eq:Mxstructuralbound} and Chebyshev's inequality, $M_2(B_{R}(p), K, N, H, D, T, \eps)$ is chosen sufficiently large so that $\haus^N(B_{R}(p) \setminus E_3) \leq  \eps/2$.

Define $S'$ to be the set of density points of $E := E_1 \cap E_2 \cap E_3$ and set $M_3 := \max\{(C_1^2 e^{2M_1})^{\frac{1}{N-2}}, 1\}$.\\ 
For each $x \in S'$, let $r_x' > 0$ be sufficiently small so that 
\begin{equation}
\frac{\haus^N(E \cap B_{r}(x))}{\haus^N(B_{r}(x))} \geq \frac{1}{2}\, , \quad\text{for any $r \leq r_x'$}\, .
\end{equation} 
Then we choose $r_x := \min\{r_x', \frac{R}{M_3}\}$. Notice that, for any $r \leq r_x$, any $t \in [0,T]$ and $\haus^N$-a.e. $y \in E \cap B_{r}(x)$,
\begin{equation}\label{eq:badlipbound}
\dist(\XX_t(x), \XX_t(y)) \leq M_3\dist(x,y) \leq 1\, ,
\end{equation}
by \autoref{prop:esttrajbar} and \autoref{prop:equivalenceGreendistance}. 
\medskip

Fix $x \in S'$, $r \in (0, r_x]$ and $0 \leq t_1 < t_2 \leq T$. Without loss of generality, we will assume $T \leq 1$.\\ 
Define 
\begin{equation}
\omega := t_2 - t_1\, \quad\text{and}\quad \mu := \frac{1}{M_3}\omega^{\frac{1}{2(1+2N)}} \leq \frac{1}{M_3} \leq 1\, . 
\end{equation}
By the very definition of $r_x$ and since $\mu r \leq r_x$, there exists some set $E_{x,\mu r}$, which can be taken up to a set of measure $0$ equal to $S' \cap B_{\mu r}(x)$, such that
\begin{equation}\label{eq:Exmurbounds}
	\frac{\haus^N(E_{x,\mu r})}{\haus^N(B_{\mu r}(x))} \geq \frac{1}{2} \; \; \text{and}\; \; \XX_{t}(E_{x, \mu r}) \subseteq B_{M_3 \mu r}(\XX_{t}(x)) \text{ for any } t \in [0,T]\, .
\end{equation}
We will now use the trajectory of $\XX_{t_1}(E_{x,\mu r})$ under $\XX_{t_1,t_1+s}$ to keep track of the trajectory of a large subset of $B_{r}(\XX_{t_1}(x))$ under $\XX_{t_1, t_1+s}$. 

In view of \eqref{eq:semigroup property}, we may assume, up to altering $E_{x, \mu r}$ by a set of measure 0, that 
\begin{equation}\label{eq:Exmursemigroup}
	\XX_{t_1, t_1+s}(\XX_{t_1}(z)) = \XX_{t_1+s}(z)\, , \text{ for any } z \in E_{x, \mu r} \text{ and any } s \in [0, T-t_1]\, . 
\end{equation}

Using \autoref{corollary:dtcontrolcor} with $S_1 = B_{r}(\XX_{t_1}(x))$, $S_2 = \XX_{t_1}(E_{x, \mu r})$ and RLF $\XX_{t_1,t_1+\cdot}$, we have that for $\mathscr{L}^1$-a.e. $s \in [0, \omega]$, setting $t=t_1+s$ in order to simplify the notation,
\begin{align}\label{eq:finedtcontrol1}
	\begin{split}
	&\hspace{0.5cm}\frac{\di}{\di s} \int_{S_1 \times S_2} dt^{\XX_{t_1,t_1+\cdot}}_{r}(s)(y,z) \, \di (\haus^N \times \haus^N)(y,z)\\
	&\leq \int_{0}^{1} \int_{\Gamma_{r}(s)} \dist(\XX_{t_1,t}(y), \XX_{t_1,t}(z))|\nabla_{\sym}(b_{t})|(\gamma_{\XX_{t_1,t}(y), \XX_{t_1,t}(z)}(u)) \, \di (\haus^N \times \haus^N)(y,z) \, \di u\, ,
	\end{split}
\end{align}
where, by definition, $\Gamma_{r}(s) = \{(y,z)\in S_1 \times S_2 : dt^{\XX_{t_1, t_1+ \cdot}}_{r}(s)(y,z)<r\}$.

Observe that (recalling that we have set $t=t_1+s$):
\begin{itemize}
	\item[i)] for any $s \in [0,\omega]$, $\XX_{t_1,t}(S_2) = \XX_{t}(E_{x,\mu r}) \subseteq B_{M_3\mu r}(\XX_{t}(x))$, by \eqref{eq:Exmurbounds};
	\item[ii)] for any $y,z \in S_1 \times S_2$, $\dist(y,z) \leq (M_3 \mu +1)r$ since $S_2 \subseteq B_{M_3 \mu r}(\XX_{t_1}(x))$, by \eqref{eq:Exmurbounds} again.
\end{itemize}
Therefore, for any $s \in [0,\omega]$ and $(y,z) \in \Gamma_{r}(s)$,
\begin{align*}
	\dist(\XX_{t_1,t}(y), \XX_{t}(x)) &\leq \dist(\XX_{t_1, t}(y), \XX_{t_1,t}(z)) + \dist(\XX_{t_1,t}(z), \XX_{t}(x))\\
						&\leq \dist(y,z) + |\dist(\XX_{t_1, t}(y), \XX_{t_1,t}(z)) - \dist(y,z)|+ \dist(\XX_{t_1,t}(z), \XX_{t}(x))\\
						&\leq (M_3 \mu + 1)r + r + M_3 \mu r \leq 4r \leq 4R\, . 
\end{align*}
Hence 
\begin{equation*}
(\XX_{t_1, t}, \XX_{t_1, t})(\Gamma_{r}(s)) \subseteq B_{4r}(\XX_{t}(x))\times B_{4r}(\XX_{t}(x)) \subseteq B_{4R}(\XX_{t}(x))\times B_{4R}(\XX_{t}(x))\, .
\end{equation*}

Now we can estimate, starting from \eqref{eq:finedtcontrol1},
\begin{align}\label{eq:finedtcontrol2}
	\begin{split}
	&\hspace{0.5cm}\int_{0}^{1} \int_{\Gamma_{r}(s)} \dist(\XX_{t_1,t}(y), \XX_{t_1,t}(z))|\nabla_{\sym}(b_{t})|(\gamma_{\XX_{t_1,t}(y), \XX_{t_1,t}(z)}(u)) \, \di (\haus^N \times \haus^N)(y,z) \, \di u\\
	&\leq e^{DT} \int_{0}^{1} \int_{(\XX_{t_1,t}, \XX_{t_1,t})(\Gamma_{r}(s))} \dist(y,z)|\nabla_{\sym}(b_{t})|(\gamma_{y,z}(u)) \, \di (\haus^N \times \haus^N)(y,z) \, \di u\\
	&\leq e^{DT} \int_{0}^{1} \int_{B_{4r}(\XX_{t}(x)) \times B_{4r}(\XX_{t}(x))} \dist(y,z)|\nabla_{\sym}(b_{t})|(\gamma_{y,z}(u)) \, \di (\haus^N \times \haus^N)(y,z) \, \di u\\
	&\leq e^{DT}c(K,N)r\haus^N(B_{4r}(\XX_{t}(x)))\int_{B_{4r}(\XX_{t}(x))}|\nabla_{\sym}b_{t}| \, \di \haus^N\\
	&\leq e^{DT}c(K,N)r\left(\haus^N(B_{4r}(\XX_{t}(x)))\right)^2\fint_{B_{4r}(\XX_{t}(x))}|\nabla_{\sym}b_{t}| \, \di \haus^N\\
	&= c(K, N, D, T)r\left(\haus^N(B_{4r}(\XX_{t}(x)))\right)^2\fint_{B_{4r}(\XX_{t}(x))}|\nabla_{\sym}b_{t}| \, \di \haus^N\, ,
	\end{split}
\end{align}
where we used \eqref{eq:divergence bound} for the second line and the Cheeger-Colding segment inequality (see \cite{CheegerColding96} for the original formulation and \cite{VonRenesse08}, \cite[Theorem 3.22]{Deng20} for this framework) for the fourth line.\\ 
Therefore,
\begin{align}\label{eq:helpful2}
	\begin{split}
		&\hspace{0.5cm}\int_{S_1 \times S_2} dt^{\XX_{t_1, t_1+\cdot}}_{r}(\omega)(y,z) \, \di (\haus^N \times \haus^N)(y,z)\\
		&= \int_{0}^{\omega} \left[\frac{\di}{\di s} \int_{S_1 \times S_2} dt^{\XX_{t_1,t_1+ \cdot}}_{r}(s)(y,z) \, \di (\haus^N \times \haus^N)(y,z)\right] \, \di s\\
		&\leq \int_{0}^{\omega}\left[cr\left(\haus^N(B_{4r}(\XX_{t_1+s}(x))\right)^2 \fint_{B_{4r}(\XX_{t_1+s}(x))}|\nabla_{\sym}b_{t_1+s}| \, \di \haus^N\right] \, \di s\\
		&\leq c(B_R(p), K, N, H, D, T, \eps)r\left(\haus^N(B_{r}(x))\right)^2  \int_{0}^{\omega}\fint_{B_{4r}(\XX_{t_1+s}(x))}|\nabla_{\sym}b_{t_1+s}| \, \di \haus^N \, \di s\\
		&\leq cr\left(\haus^N(B_{r}(x))\right)^2\sqrt{M_2}\sqrt{\omega} = c(B_R(p), K, N, H, D, T, \eps)r\left(\haus^N(B_{r}(x))\right)^2\sqrt{\omega}\, .
	\end{split}
\end{align}
Above, we used the Bishop-Gromov inequality and $N$-Ahlfors regularity of noncollapsed $\RCD(K,N)$ spaces for the fourth line and Cauchy-Schwarz, the fact that $x \in S' \subseteq E_3$, the definition of $M_2$ and that $4r \leq 4$ for the fifth line. 

Using \eqref{eq:Exmurbounds}, \eqref{eq:divergence bound} and the Bishop-Gromov inequality, we have that
\begin{equation}\label{eq:helpful}
	\frac{\haus^N(S_2)}{\haus^N(B_{r}(x))} \geq \frac{e^{-DT}\haus^N(E_{x,\mu r})}{\haus^N(B_{r}(x))} \geq \frac{e^{-DT}\haus^N(B_{\mu r}(x))}{2\haus^N(B_{r}(x))} \geq c(K, N, D, T)\mu^{N}\, .
\end{equation} 
Combining \eqref{eq:helpful} with \eqref{eq:helpful2}, we can find $z \in S_2 = \XX_{t_1}(E_{x,\mu r})$ so that
\begin{align}
\nonumber	\int_{S_1} dt^{\XX_{t_1, t_1+\cdot}}_{r}(\omega)(y,z) \, \di \haus^N(y) 
	&\leq c(B_R(p), K, N, H, D, T, \eps)r\haus^N(B_{r}(x))\mu^{-N}\sqrt{\omega}\\
\nonumber											&= cr\haus^N(B_{r}(x))(\frac{1}{M_3}\omega^{\frac{1}{2(1+2N)}})^{-N}\sqrt{\omega}\\
											&= c(B_R(p), K, N, H, D, T, \eps)r\haus^N(B_{r}(x))\omega^{\frac{1+N}{2(1+2N)}},\label{eq:boundintdist}
\end{align}
where in the last line we used the dependence of $M_3$.\\ 
Using again the $N$-Ahlfors regularity of $X$, which says that the measure of $B_{r}(\XX_{t_1}(x))$ is comparable to that of $B_{r}(x)$,, \eqref{eq:boundintdist} and Chebyshev's inequality, we can find some subset $S_1' \subseteq S_1 = B_{r}(\XX_{t_1}(x))$ with
\begin{equation}\label{eq:fineS1'bounds1}
\frac{\haus^N(S_1')}{\haus^N(B_r(\XX_{t_1}(x)))} \geq 1 - \mu^N
\end{equation}
and
\begin{align}\label{eq:fineS1'bounds2}
\begin{split}
dt^{\XX_{t_1,t_1+\cdot}}_{r}(\omega)(y,z) &\leq \frac{cr\omega^{\frac{1+N}{2(1+2N)}}}{\mu^N}\\
					&= c(B_R(p), K, N, H, D, T, \eps)r\omega^{\frac{1}{2(1+2N)}}=c\mu r < r \, ,
\end{split}
\end{align}
for any $y \in S_1'$ and any sufficiently small $\omega$ depending on $B_R(p), K, N, H, D, T$ and $\eps$.

Since for any $y \in S_1'$ we also have $\dist(y,z) \leq (M_3\mu+1)r$, we can estimate
\begin{align}\label{eq:fineS1'bounds3}
\begin{split}
	\dist(\XX_{t_1,t_2}(y),\XX_{t_2}(x)) &\leq \dist(\XX_{t_1, t_2}(y), \XX_{t_1, t_2}(z)) + \dist(\XX_{t_1, t_2}(z), \XX_{t_2}(x))\\
				&\leq \dist(y,z) + |\dist(\XX_{t_1, t_2}(y), \XX_{t_1, t_2}(z)) - \dist(y,z)| +  \dist(\XX_{t_1, t_2}(z), \XX_{t_2}(x))\\
				&\leq \dist(y,z) + dt^{\XX_{t_1,t_1+\cdot}}_{r}(\omega)(y,z) +  \dist(\XX_{t_1, t_2}(z), \XX_{t_2}(x))\\
				&\leq (M_3\mu+1)r+c\mu r + M_3\mu r\\
				&= (1+c(B_R(p), K, N, H, D, T, \eps)\mu)r\, .
\end{split}
\end{align}
Above we used that $dt^{\XX_{t_1,t_1+\cdot}}_{r}(\omega)(y,z) < r$ for $y \in S_1'$ in the third line and the definition of $\mu$ and the dependence of $M_3$ in the last line. In other words, $B_{r}(\XX_{t_1,t_2}(S_1')) \subseteq B_{(1+c\mu)r}(\XX_{t_2}(x))$.\\ 
This inclusion immediately gives the following volume estimate:
\begin{align*}
	\frac{\haus^N(B_{r}(\XX_{t_1}(x))}{\haus^N(B_{r}(\XX_{t_2}(x)))} 
	&\leq \frac{1}{1-2\omega}\frac{\haus^N(S_1')}{\haus^N(B_r(\XX_{t_2}(x)))}\\
										&\leq \frac{1}{1-\omega}e^{D\omega}\frac{\haus^N(\XX_{t_1,\omega'}(S_1'))}{\haus^N(B_r(\XX_{t_2}(x)))}\\
										&\leq \frac{1}{1-\omega}e^{D\omega}\frac{\haus^N(B_{(1+c\mu)r}(\XX_{t_2}(x)))}{\haus^N(B_r(\XX_{t_2}(x)))}\\
										&\leq \frac{1}{1-\omega}e^{D\omega}(1+c(K,N,R)c\mu)^{N}\\
										&= \frac{1}{1-\omega}e^{D\omega}(1+c\omega^{\frac{1}{2(1+2N)}})^{N},
\end{align*}
where we used \eqref{eq:fineS1'bounds1} for the first line, \eqref{eq:divergence bound} for the second line, and the Bishop-Gromov inequality for the fourth line.\\
This yields a bound of the form
 \begin{equation}\label{eq:ucvolumeratio+}
\frac{\haus^N(B_r(\XX_{t_1}(x)))}{\haus^N(B_r(\XX_{t_2}(x)))}\le 1+ g(\omega)\, ,\quad\text{for any $0<r<r_x$ and any $0\le t_1\le t_2\le T$}\, ,
\end{equation}
where $\omega=t_2-t_1$ and $g$ is a modulus of continuity independent of $r$.
\medskip

To establish the bound in the other direction, we will consider the RLF $(\YY_s)$ associated with the vector field $(-b_{t_2 - s})_{s \in[0, t_2]}$, basically reversing time in the argument.

By \cite[Proposition 3.12]{Deng20}, we may alter $E_{x, \mu r}$ up to a set of measure $0$ so that for any $z \in E_{x, \mu r}$, for any $s \in [0,t_2]$, we have $\YY_{s}(\XX_{t_2}(z)) = \XX_{t_2-s}(z)$. As such, $\YY_s(\XX_{t_2}(E_{x, \mu r})) = \XX_{t_2-s}(E_{x, \mu r})$. In particular, $\YY_s(\XX_{t_2}(E_{x, \mu r})) \subseteq B_{M_3 \mu r}(\XX_{t_2-s}(x))$ for any $s \in [0, t_2]$.

Then we can use the trajectory of $\XX_{t_2}(E_{x, \mu r})$ under $\YY_s$ to control the trajectory of a large portion of $B_{r}(\XX_{t_2}(x))$ under $\YY_s$ as we did previously. This will obtain a lower bound of the form
\begin{equation}
\frac{\haus^N(B_{r}(\XX_{t_1}(x)))}{\haus^N(B_{r}(\XX_{t_2}(x)))}\ge 1+ g(\omega)\, ,\quad\text{for any $0<r<r_x$ and any $0\le t_1\le t_2\le T$}\, ,
\end{equation}
for another modulus of continuity $g$ independent of $r$, which completes he proof of \eqref{eq:ucvolumeratio}. 

\medskip

Passing to the limit in \eqref{eq:ucvolumeratio} as $r\downarrow 0$, we conclude that $[0,T]\ni t \mapsto \theta(\XX_t(x))$ is continuous for $\haus^N$-a.e. $x\in B_R(p)$, where $\theta(x)$ denotes the density at $x$, see \eqref{density}.\\
Moreover, combining the bounded compressibility \eqref{eq:compressibility} with Fubini's theorem, we know that for $\haus^N$-a.e. $x \in B_{R}(p)$, $\XX_{t}(x)$ is a regular point for $\mathscr{L}^1$-a.e. $t \in [0,T]$. Equivalently, $\theta(\XX_t(x))=1$ for $\leb^1$-a.e. $t\in[0,T]$.\\
Hence $\theta(\XX_t(x))=1$ for any $t\in[0,T]$ and therefore $\XX_t(x)$ is a regular point for any $t\in [0,T]$.
\end{proof}



\subsection{A simple approach for spaces without boundary}
In this section we present a simpler proof of \autoref{measurecontinuity} in the case of spaces without boundary. It is based on the principle that the bounded compressibility assumption, coupled with an integrability bound on the vector field, is enough to guarantee avoidance of sets with codimension two, in a strong enough sense. As such, it does not require a careful analysis of the regularity of Lagrangian flows, but a better understanding of the fine structure of noncollapsed $\RCD(K,N)$ spaces. Unfortunately, it is not suited for dealing with codimension one singularities, such as boundary points.
\medskip

Let us recall that any noncollapsed $\RCD(K,N)$ m.m.s. $(X,\dist,\haus^N)$ can be decomposed as $X=\mathcal{R}\cup \mathcal{S}$ where $\mathcal{R}=\{ x\in X:\, \theta(x)=1\}$ is the regular set, while $\mathcal{S}$ stratifies as
\begin{equation}
	\mathcal{S}^0\subset \ldots \subset \mathcal{S}^{N-2}\subset \mathcal{S}^{N-1} = \mathcal{S}\, ,
\end{equation}  
where $x\in\mathcal{S}^k$ if and only if no tangent cone of $(X,\dist,\haus^N)$ at $x$ splits a factor $\setR^{k+1}$.

Moreover, we say that $(X,\dist,\haus^N)$ has empty boundary (in formula $\partial X=\emptyset$) if $\mathcal{S}^{N-1}\setminus \mathcal{S}^{N-2}=\emptyset$, see \cite{BrueNaberSemola20} after \cite{DePhilippisGigli17, KapovitchMondino19}.
\medskip

We are going to need the notion of \emph{quantitative singular stratum}, as introduced in \cite{CheegerNaber13} (see also \cite{AntonelliBrueSemola} for the present framework). 
\begin{definition}
For any $\eta>0$, let us define the $k^{th}$-effective stratum $\mathcal{S}^k_{\eta}$ by
\begin{equation}\label{eq:quantsing}
\mathcal{S}^{k}_{\eta}:= \set{y|\dist_{GH}(B_s(y),B_s\left((0,z^*)\right))\ge\eta s\quad\text{for all } \setR^{k+1}\times C(Z)\quad\text{and all } 0< s\le1}\, ,
\end{equation}
where $B_s\left((0,z^*)\right)$ denotes the ball in $\setR^{k+1}\times C(Z)$ centered at $(0,z^*)$ with radius $s$ and $C(Z)$ denotes any metric measure cone over an $\RCD(N-k-3,N-k-2)$ m.m.s. $(Z,\dist_Z,\haus^{N-k-1})$.

\end{definition}

For the sake of clarity, let us also recall that 
\begin{equation}\label{eq:stratumeta}
\mathcal{S}^k=\bigcup_{\eta>0}\mathcal{S}^k_{\eta}\, .
\end{equation}

\medskip
The following argument is based on \cite{Aizenman}.

\begin{proposition}
	Let $(X,\dist,\haus^N)$ be an $\RCD(K,N)$ m.m.s. and let $p>2$. Any regular Lagrangian flow $\XX$ of a velocity field $b \in L^1([0,T]; L^p(TX))$ satisfies the following property: for $\haus^N$-a.e. $x\in X$ it holds $\XX_t(x) \in X\setminus \mathcal{S}^{N-2}$ for any $t\in [0,T]$.
	
    In particular, if $\partial X=\emptyset$, then for $\haus^N$-a.e. $x\in X$ it holds that $\XX_t(x)$ is a regular point for any $t\in[0,T]$.	
\end{proposition}

\begin{proof}
	Let $\eta>0$, $\eps>0$, $r_0>0$, $R\ge 1$ and $M>1$ be fixed.\\ 
	Let $\mathcal{S}^{N-2}_\eta$ be the quantitative singular strata of codimension two (see \eqref{eq:quantsing}) and let $\dist_{ \mathcal{S}^{N-2}_\eta }$ denote the distance function from $\mathcal{S}^{N-2}_\eta $. In order to ease the notation we shall abbreviate $\dist_{\eta}:=\dist_{ \mathcal{S}^{N-2}_\eta }$.
	
	Let us assume $\eps \le r_0/2$ and set
	\[
	\tau_\eps(x) := 
	\begin{cases}
		\sup\set{t\in [0,T]\ :\ \dist_{\eta }(\XX_s(x))>\eps\quad \forall\, s\in [0,t]} & \text{if } \dist_{\eta}(x)>\eps
		\\
		\eps & \text{if } \dist_{\eta }(x)\le \eps,
	\end{cases} 
	\]
	and
	\begin{equation}\label{eq:defF}
	F := \set{x\in B_R(p) :\ \XX_t(x)\in B_{RM}(p)\ \forall\, t\in [0,T]\ \text{and}\ \dist_{ \eta }(x)\ge r_0, \ \tau_\eps(x)<T}\, ,
	\end{equation}
	for a given $p\in X$.
\medskip
	
	For any nonnegative function $f\in C^{\infty}(\setR)$ such that $f\equiv 0$ on $[r_0,\infty)$, using that $|\nabla \dist_{\eta}|=1$-a.e. and the bounded compressibility \eqref{eq:compressibility}, we can compute
	\begin{align}
	\nonumber	f(\eps)\haus^N(F)  & = \int_F \abs{f\circ \dist_{ \eta }(\XX_{\tau_\eps(x)}(x)) - f\circ \dist_{ \eta }(x)} \di \haus^N(x)\\
	\nonumber	& \le \int_F \int_0^{\tau_\eps(x)} |b_s|(\XX_s(x)) |f'\circ \dist_{ \eta }|(\XX_s(x)) \di s \di \haus^N(x)\\
		& \le L \int_0^T \int_{B_{RM}(p)} |b_s| |f'\circ\dist_{ \eta }| \di \haus^N \di s\, . \label{eq:estuse}
	\end{align}	
	A simple approximation argument allows us to consider in \eqref{eq:estuse} the test function
\begin{equation}
f(y):=
\begin{cases}
\log(r_0/y) & \text{if $y<r_0$}\\
0 & \text{if $y\ge r_0$}\, .
\end{cases}
\end{equation}	
Then we obtain
	\begin{align}
\nonumber		\log(r_0/\eps) \haus^N(F) & \le L \int_0^T \int_{\set{\dist_{\eta}  \le r_0}\cap B_{RM}(p)} |b_s|(x) \frac{1}{\dist_{\eta}(x) } \di \haus^N(x) \di s\\
\nonumber		& \le L\left( \int_0^T \norm{b_s}_{L^p}\di s\right) \left(  \int_{\set{\dist_{ \eta}  \le r_0}\cap B_{RM}(p)} \dist_{ \eta }^{-p'}  \di \haus^N \right)^{1/p'}\\
		& \le L \norm{b}_{L^1(L^p)} \left(    \int_{r_0^{-p'}}^\infty \haus^N(\set{\dist_{\eta } < \lambda^{-1/p'}}\cap B_{RM}(p)) \di \lambda  \right)^{1/p'}\, ,\label{eq:integralbound}
	\end{align}
where $1/p+1/p'=1$ and we applied H\"older's inequality at the second line and Cavalieri's formula at the third one.

Observe that, by \cite[Theorem 2.4]{AntonelliBrueSemola} (see also eq. (2.6) therein), we can bound
\begin{equation}
\haus^N\left(\set{\dist_{\eta } < \lambda^{-1/p'}}\cap B_{RM}(p)\right) \le c(K,N,MR,\eta,r_0,p) \lambda^{-\frac{2-\eta}{p'}}\, ,\quad\text{for any $\lambda>r_0^{-p'}$}\, .
\end{equation}
Since by assumption $p>2$, it holds that $p'<2$. Hence, if $\eta<\eta_0$, we have $(2-\eta)/p'>1$. Therefore
\begin{equation}
C(\eta):= \int_{r_0^{-p'}}^\infty \haus^N\left(\set{\dist_{\eta } < \lambda^{-1/p'}}\cap B_{RM}(p)\right) \di \lambda  <\infty\, .
\end{equation}
In particular, by \eqref{eq:integralbound}, we obtain that for $\eta<\eta_0$,
\begin{equation}
\log(r_0/\eps) \haus^N(F) \le L\norm{b}_{L^1(L^p)}C(\eta)\, ,
\end{equation}
independently of $\eps$.

	Letting $\eps\downarrow 0$, we deduce that, for any $\eta<\eta_0$,
	\[
	\haus^N(\set{x\in B_R(p)\ :\ \XX_t(x)\in B_{RM}(p)\ \forall t\in [0,T]\ \text{and}\  \XX_t(x)\in \mathcal{S}^{N-2}_\eta \quad \text{for some $t\in [0,T]$}}) = 0\, ,
	\]
	which easily gives the sought conclusion, taking into account \eqref{eq:stratumeta} and letting $M\to \infty$.
\end{proof}

\section{Proof of \autoref{th:main}}\label{sec:proof1}

The general strategy will be to start from \autoref{prop:esttrajbar} and turn it into an infinitesimal estimate for the lower/upper approximate slopes of the RLF relying on \autoref{corollary:Green convergence}. A priori, such an estimate would involve the ratio between the densities at the two points connected by the RLF, and we will use \autoref{measurecontinuity} to get rid of this dependence. 

In the end we will show how the technical \autoref{Assumption} can be removed, via a tensorization argument that has already been used in \cite{BrueSemola18,BrueSemolaahlfors}.
\medskip

We start with two preliminary lemmas.

\begin{lemma}\label{lemma:Lusincontinuity}
	Let $(X,\dist,\meas)$ be a locally compact metric space endowed with a $\sigma$-finite reference measure, and let $f\in L^1([0,1]\times X)$. Then, for any $\eps>0$, there exists a Borel set $E\subset X$ with $\meas(X\setminus E)<\eps$ such that, for any $t\in [0,1]$, the function $x\mapsto\int_0^tf_r(x)\di r$ is continuous in $E$.	
\end{lemma}
\begin{proof}
	We assume without loss of generality that $(X,\dist)$ is compact. The general case can be handled writing $X$ as a countable union of compact sets $K_n$ with finite measure, applying the construction described below to find good sets $E_n\setminus K_n$ such that $\meas(K_n\subset E_n)\le \eps/2^n$ and setting $E:=\cup_n E_n$.
\medskip
	
	Let $(f^n_r)_{n}\subset C(X,\dist)$ such that $\lim_{n\to \infty}\int_0^1\norm{f_r^n-f_r}_{L^1}\di r=0$. Up to extract a subsequence, for $\meas$-a.e. $x\in X$ we have
	\begin{equation*}
		\lim_{n\to \infty}\abs{\int_0^tf^n_r(x)\di r-\int_0^tf_r(x)\di r }\le \lim_{n\to \infty} \int_0^1|f^n_r(x)-f_r(x)|\di r=0\, ,
		\quad \text{for any $t\in [0,T]$}\, .
	\end{equation*}
	By Egorov theorem we can find a closed set $E$ such that $\meas(X\setminus E)<\eps$ and 
	\begin{equation*}
		\lim_{n\to \infty}\sup_{x\in E}\int_0^1|f^n_r(x)-f_r(x)|\di r\to 0\, .
	\end{equation*}
	The conclusion follows recalling that uniform limits of continuous functions are continuous. 
\end{proof}

Thanks to \autoref{lemma:doublymaximal} we will get the expected factor $t$ at the exponent in the bounds for the slope of regular Lagrangian flows of time independent Sobolev vector fields, see \eqref{eq:mainresultautonomous}. Independence of time is a crucial assumption for its proof to work.

\begin{lemma}\label{lemma:doublymaximal}
	Let $(X,\dist,\meas)$ be an $\RCD(K,N)$ m.m.s..
	Let $g\in L^2(X,\meas)$ be nonnegative, $b\in H^{1,2}_{C,s}(TX)\cap L^{\infty}(TX)$ with $\norm{\div b}_{L^{\infty}}\le D$ and let $\XX_t$ be the unique Regular Lagrangian flow of $b$. Let us set
	\begin{equation}\label{eq:introH}
		h(x):=\sup_{0<s\le T}\frac{1}{s}\int_0^s g(\XX_r(x))\di r\, .
	\end{equation}
	Then $\norm{h}_{L^2}\le C(D,T)\norm{g}_{L^2}$.
\end{lemma}

\begin{proof}
	Let us set
	\begin{equation*}
		h_t(x):=\sup_{0<s\le T}\frac{1}{s}\int_{t}^{t+s}g(\XX_r(x))\di r\, ,	
		\quad\text{for any $t\in [0,T]$ and any $x\in X$}\, .
	\end{equation*}
	Notice that the weak semi-group property \eqref{eq:semigroup property} gives, for any $t\in [0,T]$,
	\begin{equation}\label{z22}
		h_t(x)=\sup_{0<s\le T}\frac{1}{s}\int_{0}^{s}g(\XX_{r+t}(x))\di r=
		h(\XX_t(x))\, ,\quad\text{for $\meas$-a.e. $x\in X$\, .
		}
	\end{equation}
	Let us now apply the $L^2$-maximal estimate to the function $t\mapsto h_t(x)$, getting
	\begin{equation}\label{z11}
		\int_0^T h_t(x)^2 \di t\le C \int_0^{2T} g(\XX_t(x))^2\di t\, ,
		\quad \text{for any $x\in X$}\, ,
	\end{equation}
	where $C>0$ is a numerical constant. 
	
	Integrating both sides of \eqref{z11} with respect to $\meas$ and using \eqref{eq:divergence bound}, \eqref{z22}, we get
	\begin{align*}
		Te^{-DT} \int_X h^2 \di \meas\le &
		\int_0^T\int_X (h(\XX_t(x)))^2\di \meas(x)\di t\\
		=& \int_0^T\int_X (h_t(x))^2\di \meas(x)\di t\\
		\le & 2CTe^{Dt} \int_X g(x)^2\di\meas(x)\, .
	\end{align*} 
\end{proof}

\begin{proof}[Proof of \autoref{th:main}]
	Let us first prove the theorem under the additional \autoref{Assumption}, we will explain at the end how to get rid of this assumption.
	\medskip

	Fix any $0\le s<T$ and $\eps>0$. By \autoref{lemma:Lusincontinuity} we can find a Borel set $E_1\subset B_R(p)$ with $\meas(B_R(p)\setminus E_1)\le \eps$ and such that 
\begin{equation}	
	\int_s^t g_r(\XX_{s,r}(\cdot))\di r\restr_{E_1}
\end{equation}
	is continuous, for any $t\in [s,T]$.\\ 
	Set $E_2:=\set{g_s'\le 1/\eps}$, where $g_s'$ is as in \eqref{eq:LusinLipschitzflows}. Then let us take $x\in E_1\cap E_2$ such that $E_1\cap E_2$ is of density one at $x$ and there exists $E_3\subset B_R(p)$ with $\haus^N(B_R(p)\setminus E_3)=0$ for which $(x,y)$ satisfies \eqref{eq:est1bar} for any $y\in E_3$.\\
	Notice that, taking the union for $\eps\in (0,1)$, the sets of points $x\in B_R(p)$ selected in this way has full measure in $B_R(p)$. Therefore it is enough to check \eqref{eq:main} for these points.
	
	To do so, let us set $E:=E_1\cap E_2\cap E_3$. Notice that $E$ has density one at $x$ and $\XX_{s,t}\restr_E$ is Lipschitz for any $t\in [s,T]$, by \eqref{eq:LusinLipschitzflows}. Applying \autoref{prop:esttrajbar} and taking into account the continuity of $x\mapsto\int_s^t g_r(\XX_{s,r}(x))\di r$ on $E$, for any $t\in[s,T]$, we deduce
	\begin{align*}
		e^{-2\int_s^t g_r(\XX_{s,r}(x))\di r}\le \liminf_{y\in E,\ y\to x}&\frac{\dist_{G^{\lambda}}(\XX_{s,t}(x),\XX_{s,t}(y))}{\dist_{G^{\lambda}}(x,y)}\\
		&\le
		\limsup_{y\in E,\ y\to x}\frac{\dist_{G^{\lambda}}(\XX_{s,t}(x),\XX_{s,t}(y))}{\dist_{G^{\lambda}}(x,y)}
		\le e^{-2\int_s^t g_r(\XX_{s,r}(x))\di r}\, .
	\end{align*}
	Using \autoref{corollary:Green convergence} we get
	\begin{align}\label{z5}
 	&\limsup_{y\in E,\ y\to x} \frac{\dist_{G^{\lambda}}(\XX_{s,t}(x),\XX_{s,t}(y))}{\dist_{G^{\lambda}}(x,y)}\\
		&=\limsup_{y\in E,\ y\to x} 
		\left( \frac{\dist(\XX_{s,t}(x),\XX_{s,t}(y))}{\dist(x,y)}\right)^{N-2}\frac{\dist(x,y)^{N-2}G^{\lambda}(x,y)}{\dist(\XX_{s,t}(x),\XX_{s,t}(y))^{N-2}G^{\lambda}(\XX_{s,t}(x),\XX_{s,t}(y))}\\
		&=\limsup_{y\in E,\ y\to x} 
		\left( \frac{\dist(\XX_{s,t}(x),\XX_{s,t}(y))}{\dist(x,y)}\right)^{N-2}\frac{\theta(\XX_{s,t}(x))}{\theta(x)}\, .
	\end{align}
	An analogous conclusion holds for the liminf. This gives \eqref{eq:main}, up to replacing $g_r$ with $(N-2)g_r$ and up to the ratio between densities along the trajectory.
	We can now get rid of the term $\theta(\XX_{s,t}(x))/\theta(x)$ in \eqref{z5} thanks to \autoref{measurecontinuity}. In this way we obtain \eqref{eq:main}.
\medskip
	
	In the case of vector fields independent of time, the second conclusion of \autoref{th:main}, namely \eqref{eq:mainresultautonomous}, directly follows from \eqref{eq:main} and \autoref{lemma:doublymaximal}.

\medskip	
    To conclude, let us explain how to get rid of \autoref{Assumption}. 
	We rely on a tensorization argument similar to the one presented in \cite{BrueSemola18,BrueSemolaahlfors}.\\ 
	Let us define $Y=X\times\setR^3$, with product metric measure structure $(Y,\dist_Y,\meas_Y)$. It is easy to verify that $(Y,\dist_Y,\meas_Y)$ verifies \autoref{Assumption}. Then let us consider $v\in L^2([0,T];H_{C,s,{\mathrm{loc}}}^1(TY))$ acting as $v\cdot \nabla (fg) = g v\cdot \nabla f$ for any $f\in \Lip(X)$, $g\in \Lip(\setR^3)$. We shall avoid stressing the dependence fo the various differential operators appearing on the reference metric measure space since there is no risk of confusion. We refer to \cite{GigliRigoni20} for a recent throughout study of second order calculus on product spaces.	\\ 
	One can easily check that $\boldsymbol{Z}_t(x,h)=(\XX_t(x),h)$ for $(x,h)\in Y$, is a RLF associated to $v$. We aim at applying the regularity estimate to $\boldsymbol{Z}_t$ over $(Y,\dist_Y,\meas_Y)$ in order to get the sought estimate for $\XX_t$ on $(X,\dist,\haus^N)$.\\ 
	To this aim we need to slightly modify $v$ to make its support compact. Fix a constant $M>1$ to be made precise later and a smooth cut off function $\phi\in C^\infty(\setR^3)$ satisfying $\phi\equiv 1$ in $B_{RM}(0)$ and $\phi\equiv 0$ in $\setR^3\setminus B_{2RM}(0)$. Then we set $v'=\phi v$. Notice that $v'\in L^2([0,T];H_{C,s}^1(TY))$ and $v',\div v'\in L^\infty$. Moreover, denoting by $\boldsymbol{Z}'$ the RLF of $v'$ it holds $\boldsymbol{Z}'(t,x,h)=\boldsymbol{Z}(t,x,h)$ for $\haus^{N}\times \leb^3$-a.e. $(x,h)\in B_R(0)\times (-1,1)$ and any $t\in [0,T]$, provided $M$ is big enough.
    
 To conclude, we can apply a variant of the argument presented in the first part of the proof to $v'$ and $\boldsymbol{Z}'$. More precisely, in \eqref{z5} we keep $h=0$ fixed and take the limsup and the liminf considering only points $y\in E\cap \left(X\times \{0\}\right)$.

\end{proof}

\section{Proof of \autoref{globallipthm}}\label{sec:proof2}
The main idea for the proof is to argue in a similar manner to \cite{ColdingNaber12,KapovitchLi18}.\\
We begin with a lemma to establish some rough estimates. Notice that the difference between this statement and what can be obtained combining \autoref{prop:esttrajbar} and \autoref{prop:equivalenceGreendistance} is that $r$ can be as large as $R$.  
As for the proof of \autoref{th:main}, in this section we will argue under the additional \autoref{Assumption}. A tensorization argument similar to the one employed for \autoref{th:main} allows to get rid of this assumption.

\begin{lemma}\label{globalliproughlem}
For any $\eps > 0$, there exist $S \subseteq B_{R}(p)$, with $\haus^N(B_R(p)\setminus S)<\eps$, and a constant $\omega_1(K, N, B_{R}(p), H, D, T, \eps)>0$ so that for any $x \in S$, $r \in (0,4R]$ and any $t_1 \in [0,T)$, we can find $A_r \subseteq B_{r}(\XX_{t_1}(x))$ with the following properties:
\begin{itemize}
	\item[i)] $\frac{\haus^N(A_r)}{\haus^N(B_r(\XX_{t_1}(x)))}\geq \frac{1}{2}$;
	\item[ii)] for any $t_2 \in (t_1, t_1+\omega_1]$, $\XX_{t_1, t_2}(A_r) \subseteq B_{4r}(\XX_{t_2}(x))$. 
\end{itemize}
\end{lemma}

\begin{proof}
Let us fix any $\eps > 0$ and choose $S$ as in proof of \autoref{measurecontinuity}. Fix $x \in S$, $r \in (0,4R]$, and $t_1 \in [0,T)$. We divide the proof of the theorem in two cases, when $r \in (0, r_x]$ and when $r \in (r_x, 4R]$, where $r_x$ is defined as in the proof of \autoref{measurecontinuity}.

\medskip

\noindent \underline{Case 1: $r \in (0, r_x]$}

The proof in this case is very similar to the argument for \autoref{measurecontinuity} and so we will skip some details. 

Let $M_2, M_3$ be as in the proof of the theorem. By definition of $r_x$, we may choose $E_{x, \frac{r}{M_3}} \subseteq B_{\frac{r}{M_3}}(x)$ so that 
\begin{equation}\label{eq:ExrM3bounds}
	\frac{\haus^N(E_{x,\frac{r}{M_3}})}{\haus^N(B_{\frac{r}{M_3}}(x))} \geq \frac{1}{2} \; \; \text{and}\; \; \XX_{t}(E_{x, \frac{r}{M_3}}) \subseteq B_{r}(\XX_{t}(x)) \text{ for any } t \in [0,T]\, .
\end{equation}
The idea is now to use the trajectory of $\XX_{t_1}(E_{x, \frac{r}{M_3}})$ under $\XX_{t_1,t_1+s}$ to control the trajectory of a large portion of $B_{r}(\XX_{t_1}(x))$, as we did before.\\ 
Let $S_1 := B_{r}(\XX_{t_1}(x))$ and $S_2 := \XX_{t_1}(E_{x, \frac{r}{M_3}})$ (possibly after a modification on a set of measure 0). After similar calculations as before (cf. with \eqref{eq:helpful2}) we obtain that, for any $\omega \in [0, T-t_1]$, 
\begin{equation*}
	\int_{S_1 \times S_2} dt^{\XX_{t_1, t_1+\cdot}}_{r}(\omega)(y,z) \, \di (\haus^N \times \haus^N)(y,z)
		\leq c(B_{R}(p),K,N, H, D, T, \eps)r\left(\haus^N(B_{r}(x))\right)^2\sqrt{\omega}\, .
	\end{equation*}
By Bishop-Gromov inequality, \eqref{eq:ExrM3bounds} and \eqref{eq:divergence bound}, arguing as in \eqref{eq:helpful}, we can find $z \in S_2$ so that 
\begin{equation}
	\int_{S_1} dt^{\XX_{t_1, \cdot}}_{r}(\omega)(y,z) \, \di \haus^N(y) \leq  c(B_{R}(p),K,N, H, D, T, \eps)r\haus^N(B_{r}(x))\sqrt{\omega}\, .
\end{equation}
Therefore, for $\omega_1'(K, N, B_{R}(p), H, D, T, \eps)$ sufficiently small and using the $N$-Ahlfors regularity of $X$, we can find a subset $A_r \subseteq B_{r}(\XX_{t_1}(x))$ such that
\begin{itemize}
	\item[i)] $\frac{\haus^N(A_r)}{\haus^N(B_r(\XX_{t_1}(x)))}\geq \frac{1}{2}$;
	\item[ii)] for any $y \in A_r$, $dt^{\XX_{t_1, t_1+\cdot}}_{r}(\omega_1')(y,z) \leq \frac{1}{2}r$.
\end{itemize}
A simple estimate with the triangle inequality and using the definition of $dt^{\XX_{t_1, \cdot}}_{r}$ and \eqref{eq:ExrM3bounds} shows that $\XX_{t_1, t_2}(A_r) \subseteq B_{4r}(\XX_{t_2}(x))$ for any $t_2 \in (t_1, t_1+\omega_1']$, as required. 
\medskip

\noindent \underline{Case 2: $r \in (r_x, 4R]$}

This case will be handled by induction/bootstrap.\\ 
Fix any $r \in (r_x, R]$. We claim that there exists $\omega_1''(K, N, B_{R}(p), H, D, T, \eps)$ so that the following holds: if for some $x \in S$, $0 \leq t_1 < t_2 \leq T$ such that $t_2 - t_1 \leq \omega_1''$, and $r \in [0,\frac{R}{4})$, there exists $A_r \subseteq B_{r}(\XX_{t_1}(x))$ with 
 \begin{enumerate}
	\item $\frac{\haus^N(A_r)}{\haus^N(B_r(\XX_{t_1}(x)))}\geq \frac{1}{2}$;
	\item $\XX_{t_1,t_1+s}(A_r') \subseteq B_{4r}(\XX_{t_1+s}(x))$ for any $s \in [0, t_2-t_1]$,
\end{enumerate}
then the same holds for the scale of $4r$. In other words, there exists $A_{4r}' \subseteq B_{4r}(\XX_{t_1}(x))$ so that 
\begin{enumerate}
	\item $\frac{\haus^N(A_{4r}')}{\haus^N(B_{4r}(\XX_{t_1}(x)))}\geq \frac{1}{2}$;
	\item $\XX_{t_1, t_1+s}(A_{4r}') \subseteq B_{16r}(\XX_{t_1+s}(x))$ for any $s \in [0, t_2-t_1]$. 
\end{enumerate}
Combining this inductive estimate with Case 1, which plays the role of the base step, is enough to prove Case 2, one can simply take $\omega_1 := \min\{\omega_1', \omega_1''\}$.\\ 
The argument to prove the claim above uses the trajectory of $A_r$ under $\XX_{t_1, t_1+s}$ to control the trajectory of most of $B_{4r}(\XX_{t_1}(x))$ under $\XX_{t_1,t_1+s}$ and is very similar to previous estimates of this type. As such, we will not repeat it. 
\end{proof}

Having established \autoref{globalliproughlem}, we will now state a finer version which is time dependent. As will be seen, this will almost immediately give \autoref{globallipthm}.
\begin{lemma}\label{globallipfinelem}
For any $\eps > 0$, there exist $S \subseteq B_{R}(p)$, with $\haus^N\left((B_R(p)\setminus S\right)<\eps$, and constants $\omega_2(K, N, B_{R}(p), H, D, T, \eps)$, $\alpha(N), \beta(N)$ and $C(K, N, B_{R}(p), H, D, T, \eps)$ such that the following holds: for any $x \in S$, $r \in (0,4R]$ and $0 \leq t_1 < t_2 \leq T$ with $t_2-t_1 \leq \omega_2$, there exists $A_r \subseteq B_{r}(\XX_{t_1}(x))$ so that
\begin{itemize}
	\item[i)] $\frac{\haus^N(A_r)}{\haus^N(B_r(\XX_{t_1}(x)))}\geq 1 - (t_2-t_1)^{\beta}$;
	\item[ii)] for any $y \in A_r$, $\dist(\XX_{t_1,t_2}(y), \XX_{t_2}(x)) \leq \dist(y, \XX_{t_1}(x)) + C(t_2 - t_1)^{\alpha}r$. 
\end{itemize}
\end{lemma}

\begin{proof}
Fix any $\eps > 0$. We will fix $\omega_2$ later but assume for the moment that it is less than $\omega_1$ from \autoref{globalliproughlem}. We again choose $S$ as in the proof of \autoref{measurecontinuity}.\\ 
Fix now any $x \in S$, $0 \leq t_1 < t_2 \leq T$ with $t_2 - t_1 \leq \omega_2$, and $r \in (0, 4R]$. Define $\omega := t_2-t_1$ and $\mu := \omega^{\frac{1}{2(1+2N)}}$.\\ 
We can apply \autoref{globalliproughlem} to find a subset $S_2 \subseteq B_{\mu r}(\XX_{t_1})$ such that 
\begin{enumerate}
	\item $\frac{\haus^N(S_2)}{\haus^N(B_{\mu r}(\XX_{t_1}(x)))}\geq \frac{1}{2}$;
	\item for any $s \in [0,\omega]$, $\XX_{t_1, t_1+s}(S_2) \subseteq B_{4\mu r}(\XX_{t_1+s}(x))$. 
\end{enumerate}
Then we can use the trajectory of $S_2$ under $\XX_{t_1,t_1+s}$ to control the trajectory of most of $B_{r}(\XX_{t_1}(x))$ under $\XX_{t_1,t_1+s}$. The computation is nearly identical to the proof of \autoref{measurecontinuity} so we will not repeat it.\\ 
This enables us (see \eqref{eq:fineS1'bounds1}, \eqref{eq:fineS1'bounds2}) to find some $A_r \subseteq B_{r}(\XX_{t_1}(x))$ with 
\begin{equation}
	\frac{\haus^N(A_r)}{\haus^N(B_{r}(\XX_{t_1}(x)))} \geq 1 - \mu^{N}\, ,
\end{equation}
and some $z \in S_2$ such that for any $y \in A_r$ and any $s \in [0, \omega]$,
\begin{equation}
	dt^{\XX_{t_1,t_1+\cdot}}_r(\omega)(y,z) \leq c(K, N, B_{R}(p), H, D, T, \eps)\mu r\, .
\end{equation}
Moreover, choosing $\omega_2(K, N, B_{R}(p), H, D, T, \eps)$ sufficiently small, we may assume that $c\mu r < r$. Using the triangle inequality and the fact that $z \in S_2$,  we find that, for any $y \in A_r$ and any $s \in [0, \omega]$,
\begin{align*}
\dist(\XX_{t_1,t_1+s}(y), \XX_{t_1+s}(x)) \leq &\, \dist(\XX_{t_1,t_1+ s}(y), \XX_{t_1,t_1+s}(z))+\dist(\XX_{t_1, t_1+s}(z), \XX_{t_1+s}(x))\\
						\leq &\, \dist(y,z)+|\dist(\XX_{t_1,t_1+ s}(y), \XX_{t_1,t_1+s}(z))-\dist(y,z)|\\
						&+\dist(\XX_{t_1, t_1+s}(z), \XX_{t_1+s}(x))\\
				\leq &\, \dist(y,\XX_{t_1}(x))+\dist(\XX_{t_1}(x),z)+dt^{\XX_{t_1,\cdot}}_r(\omega)(y,z)\\
				&+\dist(\XX_{t_1,t_1+ s}(z), \XX_{t_1+s}(x))\\
				\leq &\, \dist(y,\XX_{t_1}(x))+\mu r + c\mu r + 4 \mu r\\
				\leq &\, \dist(y,\XX_{t_1}(x)) + C(K, N, B_{R}(p), H, D, T, \eps)\mu r\, .
\end{align*}
This immediately gives the claim with $\beta = \frac{N}{2(1+2N)}$ and $\alpha = \frac{1}{2(1+2N)}$, since $\mu = (t_2-t_1)^{\frac{1}{2(1+2N)}}$. 
\end{proof}

\begin{proof}[Proof of Theorem \ref{globallipthm}]
Fix any $\eps > 0$ and the same $S$ as before. Fix any $x, y \in S$. Fix some $0 \leq t_1 < t_2 \leq T$ with $t_2 - t_1 \leq \omega_2$ given by \autoref{globallipfinelem}.\\ 
It is straightforward to check that $\XX_{t}(x) \in \overline{B_{R}(p)}$ (likewise for $y$) for any $t \in [0,T]$, since a set of positive measure in $B_R(p)$ stays arbitrarily close to $\XX_{t}(x)$ under the flow $\XX_{t}$ (by definition of $S$) and $b$ is supported in $B_{R}(p)$. 

Define $r := \dist(\XX_{t_1}(x), \XX_{t_1}(y)) \leq 2R$. Applying \autoref{globallipfinelem} to $B_{2r}(\XX_{t_1}(x))$ we can find $A_{2r}^{x} \subseteq B_{2r}(\XX_{t_1}(x))$ such that
\begin{itemize}
	\item[1x)] $\frac{\haus^N(A^{x}_{2r})}{\haus^N(B_{2r}(\XX_{t_1}(x)))}\geq 1 - (t_2-t_1)^{\beta}$;
	\item[2x)] for any $z \in A^{x}_{2r}$, $\dist(\XX_{t_1,t_2}(z), \XX_{t_2}(x)) \leq \dist(z, \XX_{t_1}(x)) + C(t_2 - t_1)^{\alpha}r$. 
\end{itemize}
Analogously, applying \autoref{globallipfinelem} to $B_{2r}(\XX_{t_1}(y))$ we can find $A_{2r}^{y} \subseteq B_{2r}(\XX_{t_1}(y))$ such that 
\begin{itemize}
	\item[1y)] $\frac{\haus^N(A_{2r}^{y})}{\haus^N(B_{2 r}(\XX_{t_1}(y)))}\geq 1 - (t_2-t_1)^{\beta}$;
	\item[2y)] for any $z \in A_{2r}^{y}$, $\dist(\XX_{t_1,t_2}(z), \XX_{t_2}(y)) \leq \dist(z, \XX_{t_1}(y)) + C(t_2 - t_1)^{\alpha}r$. 
\end{itemize}
\medskip

Let us consider the set $E := A_{2r}^{x} \cap A_{2r}^{y} \cap B_{r}(\XX_{t_1}(y))$. By Bishop-Gromov inequality, 1x) and 1y), we have that 
\begin{equation}
\frac{\haus^N(E)}{\haus^N(B_{r}(\XX_{t_1}(y)))} \geq 1 - c(K,N,R)(t_2-t_1)^{\beta}\, .
\end{equation} 
By Bishop-Gromov inequality again, $E$ is $c(K,N,R)(t_2-t_1)^{\frac{\beta}{N}}r$-dense in $B_{r}(\XX_{t_1}(y))$. In particular, there exists $z \in E$ so that 
\begin{equation}\label{eq:epsdense}
\dist(\XX_{t_1}(y),z) \leq c(t_2-t_1)^{\frac{\beta}{N}}r = c(t_2-t_1)^{\alpha}r\, , 
\end{equation}
where we used the relationship between $\alpha$ and $\beta$ from \autoref{globallipfinelem} (see the last line of the proof, in particular).\\ 
Then,  by \eqref{eq:epsdense}, 2x) and 2y), we can estimate
\begin{align*}
\begin{split}
	\dist(\XX_{t_2}(x), \XX_{t_2}(y)) &\leq \dist(\XX_{t_2}(x), \XX_{t_1,t_2}(z)) +  \dist(\XX_{t_1, t_2}(z), \XX_{t_2}(y))\\
	&\leq \dist(z, \XX_{t_1}(x)) + C(t_2 - t_1)^{\alpha}r + \dist(z, \XX_{t_1}(y)) + C(t_2 - t_1)^{\alpha}r\\
	&\leq \dist(\XX_{t_1}(x),  \XX_{t_1}(y)) + 2\dist( \XX_{t_1}(y), z) + C(t_2 - t_1)^{\alpha}r\\
	&\leq  r+C_0(K, N, B_{R}(p), H, D, T, \eps)(t_2-t_1)^{\alpha}r\, , 
\end{split}
\end{align*}
which completes the proof. 
\end{proof}

\end{document}